\documentclass{amsart}
\usepackage[all]{xy}
\usepackage{amsmath,amsfonts,amssymb,amsthm,epsfig,amscd,comment}

\newtheorem{Theorem}{Theorem}[section]
\newtheorem{Proposition}[Theorem]{Proposition} 
\newtheorem{Lemma}[Theorem]{Lemma}

\newtheorem{Corollary}[Theorem]{Corollary}

\theoremstyle{definition}
\newtheorem{Remark}[Theorem]{Remark}
\newtheorem{Example}[Theorem]{Example}

\def\P{\mathcal{P}}
\def\Q{\mathcal{Q}}
\def\k{{\Bbbk}}

\def\id{{\mathrm{id}}}
\def\l{{\lambda}}

\def\sl{\mathfrak{sl}}
\def\gl{\mathfrak{gl}}
\def\g{\mathfrak{g}}
\def\h{\mathfrak{h}}
\newcommand{\Z}{\mathbb{Z}}
\newcommand{\C}{\mathbb{C}}
\newcommand{\p}{\mathbb{P}}
\newcommand{\Cone}{\operatorname{Cone}}

\newcommand{\D}{\mathcal{D}}
\def\O{{\mathcal O}}
\def\e{{\varepsilon}}
\def\i{{\iota}}
\def\sA{{\mathcal{A}}}
\def\sE{{\mathcal{E}}}
\def\sS{{\mathcal{S}}}
\def\tsE{{\tilde{\mathcal{E}}}}
\def\sF{{\mathcal{F}}}
\def\sL{{\mathcal{L}}}
\def\sG{{\mathcal{G}}}
\def\sH{{\mathcal{H}}}
\def\sT{{\mathcal{T}}}
\def\sP{{\mathcal{P}}}
\def\bG{\mathbb{G}}
\def\pt{\mathrm{pt}}

\def\tY{\tilde{Y}}
\newcommand{\E}{\mathsf{E}}
\newcommand{\F}{\mathsf{F}}

\newcommand{\U}{\mathsf{U}}
\newcommand{\T}{\mathsf{T}}
\newcommand{\la}{\langle}
\newcommand{\ra}{\rangle}
\newcommand{\Ext}{\operatorname{Ext}}
\newcommand{\Hom}{\operatorname{Hom}}
\newcommand{\End}{\operatorname{End}}

\newcommand{\spn}{\operatorname{span}}

\DeclareMathOperator{\supp}{supp}

\begin{document}

\title[Braiding via geometric Lie algebra actions]{Braiding via geometric Lie algebra actions}

\author{Sabin Cautis}
\email{scautis@math.columbia.edu}
\address{Department of Mathematics\\ Columbia University \\ New York, NY}

\author{Joel Kamnitzer}
\email{jkamnitz@math.toronto.edu}
\address{Department of Mathematics\\ University of Toronto \\ Toronto, ON Canada}

\begin{abstract}
We introduce the idea of a geometric categorical Lie algebra action on derived categories of coherent sheaves. The main result is that such an action induces an action of the braid group associated to the Lie algebra. The same proof shows that strong categorical actions in the sense of Khovanov-Lauda and Rouquier also lead to braid group actions. As an example, we construct an action of Artin's braid group on derived categories of coherent sheaves on cotangent bundles to partial flag varieties. 
\end{abstract}

\date{\today}
\maketitle
\tableofcontents
\section{Introduction}
Let $ X $ be a smooth complex variety.  Often the derived category of coherent sheaves on $ X$, denoted $D(X)$, possesses interesting autoequivalences, not coming from automorphisms of $ X$ itself.  For example, if $ D(X) $ contains a spherical object $ \sE $, then Seidel-Thomas \cite{ST} defined a spherical twist $ T_\sE : D(X) \rightarrow D(X) $ which is a non-trivial autoequivalence. 

The notion of twists in spherical objects has been generalized by various authors (Horja \cite{Ho}, Anno \cite{A}, and Rouquier \cite{Rold}) to twists in spherical functors (a relative version).  In \cite{ckl2}, \cite{ckl3}, we (jointly with Anthony Licata and following ideas of Chuang-Rouquier \cite{CR}) defined the notion of geometric categorical $ \sl_2 $ actions as a generalization of the notion of spherical functors.  We showed that geometric categorical $\sl_2 $ actions give rise to equivalences of derived categories of coherent sheaves. 

Often autoequivalences of $D(X)$ can be organized into an action of a braid group.  Seidel-Thomas \cite{ST} showed that given a collection of spherical objects which form a type $ \Gamma $ arrangement, the spherical twists generate an action of the braid group $ B_\Gamma $.  An important example from \cite{ST} of this situation concerned the case where $ X $ is the resolution of a surface quotient singularity $\C^2/H$ and the spherical objects come from the exceptional $ \p^1$s.

Another example of a braid group action was given by Khovanov-Thomas in \cite{KT}.  They showed that $ B_n $ acts on $D(T^\star Fl(\C^n))$, the derived category of the cotangent bundle to the full flag variety, with the generators acting by spherical twists.  Our purpose in this paper is to introduce a new method of constructing braid group actions (called geometric categorical $\g $ actions), where the generators act by the equivalences coming from geometric categorical $ \sl_2 $ actions.  Roughly speaking, spherical objects are a special case of geometric categorical $\sl_2 $ actions and type $ \Gamma $ arrangements of spherical objects are a special case of geometric categorical $\g $ actions. 

To explain our motivation for this notion, let us recall that our proof that a geometric categorical $ \sl_2 $ action gives an equivalence came in two parts.  First in \cite{ckl2}, we showed that a geometric categorical $ \sl_2 $ action gives a strong categorical $\sl_2 $ action, a notion introduced by Chuang-Rouquier \cite{CR}.  We then showed in \cite{ckl3} that a strong categorical $\sl_2 $ action gives an equivalence, using an explicit complex introduced by Chuang-Rouquier \cite{CR}.  The reason for introducing the notion of geometric categorical $ \sl_2 $ action, rather than working with strong categorical $ \sl_2 $ actions, is that the axioms of the former are much easier to check in examples.

The notion of strong categorical $ \sl_2 $ action has been generalized by Rouquier \cite{Rnew} and Khovanov-Lauda \cite{kl1, kl2, kl3} to the notion of strong categorical $ \g$ action.  Hence it is natural to conjecture that a strong categorical $\g $ action gives an action of the braid group of type $ \g $ (denoted $ B_\g$).  Also it is natural to search for a notion of geometric categorical $ \g $ action which implies strong categorical $ \g $ action but which is easier to check in geometric examples. 

In this paper, we essentially accomplish these goals.  More specifically, we define the notion of geometric categorical $ \g $ action, whenever $\g$ is a simply-laced Kac-Moody Lie algebra. We then prove that a geometric categorical $ \g $ action gives rise to an action of $B_\g$ (Theorem \ref{thm:main}).  We also show that a strong categorical $ \g $ action gives an action of $ B_\g $ (Theorem \ref{thm:main2}).  This essentially answers a conjecture of Rouquier \cite{Rnew} (Rouquier has also recently proven his conjecture via a different method).  However, we do not prove that a geometric categorical $ \g $ action gives a strong categorical $ \g $ action, though we expect this to be the case (the proof should follow along the same lines as \cite{ckl2}, where we established this result for $ \g = \sl_2 $).

We give a quick example showing how resolutions of $ \C^2/H $ give geometric categorical $ \g $ actions.  In greater detail in section \ref{sec:example}, we discuss the more complicated example of a geometric categorical $ \sl_n $ action on cotangent bundles to $ n $-step partial flag varieties.  This generalizes the work of Khovanov-Thomas \cite{KT} for $ T^\star Fl(\C^n) $ and also our previous work \cite{ckl2, ckl3} on cotangent bundles to Grassmannians.

In a forthcoming paper with Anthony Licata \cite{ckl4}, we will construct geometric categorical $ \g $ actions on Nakajima quiver varieties, generalizing the two examples in this paper. Using the main result of this paper, this will provide many more examples of braid group actions.

There are also many interesting examples of strong categorical $\g $ actions, not involving coherent sheaves.  For examples, Khovanov-Lauda \cite{kl1, kl2, kl3} have considered strong categorical $ \sl_n $ actions on categories of modules over cohomology rings of partial flag varieties and Chuang-Rouquier \cite{CR} have defined strong categorical $\widehat{\sl}_p $ actions on categories of representations of the symmetric group in characteristic $p$.  Our Theorem \ref{thm:main2} can be applied to these situations to produce braid group actions.

\subsection*{Acknowledgements}
We would like to thank Pierre Baumann, Chris Brav, Mikhail Khovanov, Aaron Lauda, Anthony Licata and Rapha\"el Rouquier for helpful discussions. In the course of this work, S.C. was supported by NSF Grant 0801939 and J.K. by NSERC. We would also like to thank MSRI for excellent working conditions and hospitality. We also thank the referee for his very careful reading of this paper.

\section{Definitions and Main Results}

In this section we define the concept of a geometric categorical $\g$ action, review the construction of equivalences from strong categorical $\sl_2$ actions and state our main result (Theorem \ref{thm:main}). 

\subsection{Notation}
Fix a base field $ \k $, which is not assumed to be of characteristic 0, nor algebraically closed.

Let $\Gamma$ be a graph without multiple edges or loops and with finite vertex set $I$. 
In addition, fix the following data.
\begin{enumerate}
\item a free $\Z$ module $X$  (the weight lattice),
\item for $i \in I$ an element $\alpha_i \in X$ (simple roots),
\item for $i \in I$ an element $\Lambda_i \in X$ (fundamental weight),
\item a symmetric non-degenerate bilinear form $\la \cdot,\cdot \ra$ on $X$.
\end{enumerate}
These data should satisfy:
\begin{enumerate}
\item the set $\{\alpha_i\}_{i \in I}$ is linearly independent,
\item $C_{i,j} = \la \alpha_i, \alpha_j \ra$ (the Cartan matrix) so that $\la \alpha_i, \alpha_i \ra = 2$ and for $i \neq j$, $\la \alpha_i, \alpha_j \ra = \la \alpha_j, \alpha_i \ra \in \{0,-1\}$ depending on whether or not $i,j \in I$ are joined by an edge,
\item $\la \Lambda_i, \alpha_j \ra = \delta_{i,j}$ for all $i, j \in I$,
\item $\dim X = |I| + \mathrm{corank}(C)$, where $ C $ is the Cartan matrix associated to $\Gamma$.
\end{enumerate}

Let $ \h_\k = X \otimes_\Z \k $ and let $ \h'_{\k} = \spn(\Lambda_i) \subset \h_\k $.

Associated to $ \Gamma$, we have a Kac-Moody Lie algebra $ \g $ (defined over $\C$). Let $ B_\g $ denote the braid group of type $ \g $.  It has generators $ \sigma_i $ for $i \in I$ and relations 
\begin{gather*}
\sigma_i \sigma_j \sigma_i = \sigma_j \sigma_i \sigma_j \text{ if $ i $ and $ j $ are connected in $\Gamma $,} \\
\sigma_i \sigma_j = \sigma_j \sigma_i \text{ if $ i $ and $ j $ are not connected in $\Gamma $.}
\end{gather*}
Recall that $ B_\g $ maps to the Weyl group $W_\g $ of type $\g $ which has the same generators and relations, except that the generators square to the identity.

When we write $H^\star(X)$ we will mean the cohomology with $\k$ coefficients of $X$ as a variety over $\C$ but shifted so that it lies between degrees $-\dim(X)$ and $\dim(X)$. For example, $H^\star(\p^1) = \k[-1] \oplus \k[1]$. By convention, $H^\star(\p^{-1}) = 0$. 

\subsection{Geometric categorical $\g$ actions}\label{sec:geomcat}

In \cite{ckl2} we introduced the concept of a geometric categorical $\g$ action when $\g = \sl_2$. We now extend this definition to arbitrary simply-laced $\g$.  All varieties will be defined over $ \k $.

\subsubsection{Fourier-Mukai formalism}
We briefly recall the formalism of Fourier-Mukai (FM) kernels (see \cite{Hu}, section 5.1, for more details). All the functors which follow are derived. Let $ X, Y $ be two smooth varieties.  A FM kernel is any object $\P \in D(X \times Y)$ of the derived category of coherent sheaves on $ X \times Y $ whose support is proper over $X$ and $Y$. It defines the associated Fourier-Mukai transform
\begin{equation*}
\begin{aligned}
\Phi_\P : D(X) &\rightarrow D(Y) \\\sF &\mapsto {\pi_2}_* (\pi_1^* (\sF) \otimes \P)
\end{aligned}
\end{equation*}
FM transforms have right and left adjoints which are themselves FM transforms. In particular, the right adjoint of $ \Phi_\P $ is the FM transform with respect to $ \P_R := \P^\vee \otimes \pi_2^* \omega_X [\dim(X)] \in D(Y \times X)$ (here we use the natural isomorphism $X \times Y \xrightarrow{\sim} Y \times X$). Similarly, the left adjoint of $ \Phi_{\P} $ is the FM transform with respect to $ \P_L := \P^\vee \otimes \pi_1^* \omega_Y [\dim(Y)] $.  Here $ \P^\vee $ denotes the dual of $ \P $.

We can express composition of FM transforms in terms of their kernels.  If $ X, Y, Z $ are varieties and $\Phi_\P : D(X) \rightarrow D(Y),  \Phi_\Q : D(Y) \rightarrow D(Z) $ are FM transforms, then $ \Phi_\Q \circ \Phi_\P $ is a FM transform with respect to the kernel
\begin{equation*}
\Q * \P := {\pi_{13}}_*(\pi^*_{12}(\P) \otimes \pi^*_{23}(\Q))
\end{equation*}
where $*$ is called the convolution product. 

\subsubsection{} A {\bf geometric categorical $\g$ action} consists of the following data.

\begin{enumerate}
\item A collection of smooth varieties $Y(\l)$ for $\l \in X$.
\item Fourier-Mukai kernels 
\begin{equation*}
\sE^{(r)}_i(\l) \in D(Y(\l) \times Y(\l + r\alpha_i)) \text{ and } \sF^{(r)}_i(\l) \in D(Y(\l + r\alpha_i) \times Y(\l)).
\end{equation*}
We will usually write just $\sE^{(r)}_i$ and $\sF^{(r)}_i$ to simplify notation whenever possible. When $r=1$ we just write $\sE_i$ and $\sF_i$.  
\item For each $Y(\l)$ a flat deformation $\tY(\l) \rightarrow \h'_\k$ (where the fibre over $0 \in \h'_\k$ is identified with $Y(\l)$).  
\end{enumerate}

Denote by $\tY_i(\l) \rightarrow \mathrm{span}(\Lambda_i) \subset \h'_\k$ the restriction of $\tY(\l)$ to $\mathrm{span}(\Lambda_i)$ (this is a one parameter deformation of $Y(\l)$). 

\begin{Remark} In practice we only need the first order deformation of $\tY(\l)$, but in geometric examples there exists a natural deformation over $\h'_\k$. Replacing this deformation by the corresponding first order deformation does not change the results and arguments in the rest of the paper.  

Similarly, one can replace $\h'_\k$ by some abstract smooth base of the same dimension, $\mathrm{span}(\Lambda_i) \subset \h'_\k$ by one-dimensional subvarieties etc. But we use $\h'_\k$ to keep notation simpler and because in many examples the base is naturally isomorphic to $\h'_\k$. 
\end{Remark}

On this data we impose the following conditions. 
\begin{enumerate}
\item Each $\Hom$ space between two objects in $D(Y(\l))$ is finite dimensional. In particular, this means that $\End(\O_{Y(\l)}) = \k \cdot I$. \label{co:trivial}
\item All $\sE_i^{(r)}$s and $\sF_i^{(r)}$s are sheaves (i.e. complexes supported in cohomological degree zero). \label{co:sheaves}
\item $\sE^{(r)}_i(\l)$ and $\sF^{(r)}_i(\l)$ are left and right adjoints of each other up to shift. More precisely \label{co:adjoints}
\begin{enumerate}
\item $\sE^{(r)}_i(\l)_R = \sF^{(r)}_i(\l) [r(\la \l, \alpha_i \ra + r)]$ 
\item $\sE^{(r)}_i(\l)_L = \sF^{(r)}_i(\l) [-r(\la \l, \alpha_i \ra + r)]$.
\end{enumerate}
\item For each $ i \in I $, \label{co:EE}
$$\sH^*(\sE_i * \sE^{(r)}_i) \cong \sE^{(r+1)}_i \otimes_\k H^\star(\p^r).$$

\item If $\la \l, \alpha_i \ra \le 0$ then \label{co:EF}
$$\sF_i(\l) * \sE_i(\l) \cong \sE_i(\l-\alpha_i) * \sF_i(\l-\alpha_i) \oplus \sP$$
where $\sH^*(\sP) \cong \O_\Delta \otimes_\k H^\star(\p^{- \la \l, \alpha_i \ra - 1})$. 

Similarly, if $\la \l, \alpha_i \ra \ge 0$ then 
$$\sE_i(\l-\alpha_i) * \sF_i(\l-\alpha_i) \cong \sF_i(\l) * \sE_i(\l) \oplus \sP'$$
where $\sH^*(\sP') \cong \O_\Delta \otimes_\k H^\star(\p^{\la \l, \alpha_i \ra - 1})$. 

\item We have \label{co:EEdef}
\begin{equation} \label{eq:defEE}
\sH^*(i_{23*} \sE_i * i_{12*} \sE_i) \cong \sE_i^{(2)}[-1] \oplus \sE_i^{(2)}[2] 
\end{equation}
where $i_{12}$ and $i_{23}$ are the closed immersions
\begin{align*} 
i_{12}: Y(\l) \times Y(\l+ \alpha_i) &\rightarrow Y(\l) \times \tY_i(\l+\alpha_i) \\
i_{23}: Y(\l + \alpha_i) \times Y(\l + 2\alpha_i) &\rightarrow \tY_i(\l+\alpha_i) \times Y(\l+2\alpha_i).
\end{align*}

\item \label{co:contain} If $\la \l, \alpha_i \ra \le 0$ and $k \ge 1$ then the image of $\supp(\sE_i^{(r)}(\l-r \alpha_i))$ under the projection to $Y(\l)$ is not contained in the image of $\supp(\sE_i^{(r+k)}(\l-(r+k) \alpha_i))$ also under the projection to $Y(\l)$. 
Similarly, if $\la \l, \alpha_i \ra \ge 0$ and $k \ge 1$ then the image of $\supp(\sE_i^{(r)}(\l))$ in $Y(\l)$ is not contained in the image of $\supp(\sE_i^{(r+k)}(\l))$.

\item \label{co:EiEj} If $i \ne j \in I$ are joined by an edge in $\Gamma$ then  
$$\sE_i * \sE_j * \sE_i \cong \sE_i^{(2)} * \sE_j \oplus \sE_j * \sE_i^{(2)}$$
while if they are not joined then $\sE_i * \sE_j \cong \sE_j * \sE_i$. 

\item If $i \ne j \in I$ then $\sF_j * \sE_i \cong \sE_i * \sF_j$. \label{co:EiFj}

\item \label{co:Eidef} For $i \in I$ the sheaf $\sE_i$ deforms over $\alpha_i^\perp$ to some 
$$\tsE_i \in D(\tY(\l)|_{\alpha_i^\perp} \times_{\alpha_i^\perp} \tY(\l+\alpha_i)|_{\alpha_i^\perp}).$$

\item \label{co:Eijdef}
If $i \ne j \in I$ are joined by an edge, by Lemma \ref{lem:Eij}, there exists a unique non-zero map (up to multiple) $T_{ij}: \sE_i * \sE_j [-1] \rightarrow \sE_j * \sE_i$ whose cone we denote 
$$\sE_{ij} := \Cone \left( \sE_i * \sE_j [-1] \xrightarrow{T_{ij}} \sE_j * \sE_i \right) \in D(Y(\l) \times Y(\l+\alpha_i+\alpha_j)).$$
Then $\sE_{ij}$ deforms over $B := (\alpha_i+\alpha_j)^\perp \subset \h'_\k$ to some 
$$\tsE_{ij} \in D(\tY(\l)|_B \times_{B} \tY(\l+\alpha_i+\alpha_j)|_B).$$
\end{enumerate}

\begin{Remark}
The conditions (\ref{co:trivial}), (\ref{co:sheaves}), (\ref{co:adjoints}), (\ref{co:contain}) are technical conditions.  The conditions (\ref{co:EE}), (\ref{co:EF}), (\ref{co:EiEj}), (\ref{co:EiFj}) are categorical versions of the relations in the usual presentation of the Kac-Moody Lie algebra $\g$ (except as in \cite{ckl2}, we only impose parts of (\ref{co:EE}), (\ref{co:EF}) at the level of cohomology which is much easier to check).  The conditions (\ref{co:EEdef}), (\ref{co:Eidef}) and (\ref{co:Eijdef}) relate to the deformation.   

Notice that conditions (\ref{co:trivial}) - (\ref{co:contain}) are precisely equivalent to saying that $\{Y(\l+n\alpha_i)\}_{n \in \Z}$, together with $\sE_i$ and $\sF_i$ and deformations $\tY_i(\l+n\alpha_i)$ generate a geometric categorical $\sl_2$ action. Relations (\ref{co:EiEj}) - (\ref{co:Eijdef}) then describe how these various $\sl_2$ actions are related. 

One can compare the geometric definition above to the notion of a 2-representation of $ \g $ in the sense of Rouquier \cite{Rnew}, which in turn is very similar to the notion of an action of Khovanov-Lauda's 2-category \cite{kl1, kl2, kl3}.  In these definitions, there are functors $ \E_i, \F_i $ as well as some natural transformations $X, T$ between these functors.  The additional data of our deformations is perhaps equivalent to the additional deformation of these natural transformations.  In the case of $ \g = \sl_2 $, we were able to make this connection precise (see \cite{ckl2}).  For general $ \g $, it remains an open problem to show that a geometric categorical $\g $ action gives an action of Khovanov-Lauda or Rouquier's 2-category.  In any case, we work here with the above definition since these axioms can be checked in examples (as in section \ref{sec:example}).
\end{Remark}

\begin{Remark}
Since $ \sE_i, \sF_i $ are biadjoint (up to shift), the conditions (\ref{co:EE}), (\ref{co:EEdef}) and (\ref{co:EiEj}) immediately imply the same conditions where all $ \sE_i $ are replaced by $ \sF_i$.
\end{Remark}

\begin{Theorem} \label{thm:naive}
A geometric categorical $ \mathfrak{g} $ action gives a naive categorical $\mathfrak{g} $ action. By this we mean that the functors $ \E_i^{(r)} := \Phi_{\sE_i^{(r)}}$ and $\F_i^{(r)} := \Phi_{\sF_i^{(r)}}$ satisfy the defining relations of $\mathfrak{g}$, up to isomorphism:
\begin{gather*}
\E_i \circ \E_i^{(r)} \cong \E_i^{(r)} \circ \E_i^{(r+1)} \otimes_\C H^\star(\p^r) \text{ (and similarly with $\E$ replaced by $\F$), } \\
\F_i \circ \E_i \cong \E_i \circ \F_i \oplus \id_{Y(\l)} \otimes_\C H^\star(\p^{-\la \l, \alpha_i \ra -1}) \text{ if $\la \l, \alpha_i \ra \le 0$ and }\\
\E_i \circ \F_i \cong \F_i \circ \E_i \oplus \id_{Y(\l)} \otimes_\C H^\star(\p^{\la \l, \alpha_i \ra +1}) \text{ if $\la \l, \alpha_i \ra \ge 0$, } \\
\E_i \circ \E_j \circ \E_i \cong \E_i^{(2)} \circ \E_j \oplus \E_j \circ \E_i^{(2)} \text{ if $i,j$ are joined by an edge and} \\
\E_i \circ \E_j \cong \E_j \circ \E_i \text{ if they are not (and similarly with $ \E $ replaced by $ \F $),}\\
\F_j \circ \E_i \cong \E_i \circ \F_j \text{ if $i \ne j$}.
\end{gather*}
Hence  $\E_i, \F_i $ acting on the Grothendieck groups $ \{ K(Y(\l)) \} $ gives a representation of $ U(\g) $.
\end{Theorem}

\begin{proof}
Note that the first three statements differ from the conditions (\ref{co:EE}), (\ref{co:EF}) given in the definition, in that statements on the level of homology are turned into direct sums.  These facts are proven (with the help of the deformations) in \cite{ckl2}.
\end{proof} 

Finally, we define two more maps we will use repeatedly. The first map is 
$$\sE_i^{(r+1)} \xrightarrow{\i} \sE_i * \sE_i^{(r)} [-r] \cong \sE_i^{(r)} * \sE_i [-r] \text{ and } \sF_i^{(r+1)} \xrightarrow{\i} \sF_i * \sF_i^{(r)} [-r] \cong \sF_i^{(r)} * \sF_i [-r]$$
which includes into the lowest degree summand of the right hand side (the isomorphisms above follow from Proposition \ref{prop:EXE}). Notice that there is a unique such map (up to multiple) because by Lemma \ref{lem:Eij} we have that $\End^k(\sE_i^{(r+1)})$ is zero if $k < 0$ and one-dimensional if $k=0$ (similarly with $\End^k(\sF_i^{(r+1)})$). 

The second map is 
$$\sF_i^{(r)} * \sE_i^{(r)} (\l) \xrightarrow{\e} \O_\Delta [r(\la \l, \alpha_i + r\ra)] \text{ and } \sE_i^{(r)} * \sF_i^{(r)} (\l) \xrightarrow{\e} \O_\Delta [-r(\la \l, \alpha_i + r\ra)]$$
given by adjunction (by definition $\sE_i^{(r)}$ and $\sF_i^{(r)}$ are adjoint to each other up to shifts). This map is also uniquely defined (up to multiple). 

We say that a geometric categorical $\g$-action is {\bf integrable} if for every weight $\l$ and $i \in I$ we have $Y(\l + n\alpha_i) = \emptyset$ for $n \gg 0$ or $n \ll 0$. From hereon we assume all actions are integrable. 

\subsection{Role of the deformation}\label{sec:role}

\subsubsection{Generalities on deformations and obstructions}
Let  $ Y $ be a variety with a 1-parameter deformation $ \tY \rightarrow \mathbb{A}^1 $.  The obstruction to deforming $ \sA \in D(Y) $ is given by a map $ c(\sA) : \sA[-1] \rightarrow \sA[1] $.  By definition, $c(\sA)[1] $ is the connecting map coming from the exact triangle $ \sA[1] \rightarrow t^* t_* \sA \rightarrow \sA $, where $ t $ is the inclusion $ t : Y \rightarrow \tY$ (see \cite[appendix]{HT}).  More generally, if $ \tY \rightarrow V $ is a deformation over a $\k$-vector space $ V $, then for each $ v \in V $, we get a 1-parameter deformation over $ \spn(v) $ and we write $ c_v(\sA) : \sA[-1] \rightarrow \sA[1] $ for the obstruction map.  Let us summarize the properties of these maps.

\begin{Proposition} \label{th:propc1}
Let $ \tY \rightarrow V $ be a deformation over a $ \k$-vector space $ V $.
\begin{enumerate}
\item There is a functorial distinguished triangle $$ \sA[-1] \xrightarrow{c_v(A)} \sA[1] \rightarrow {t_v}^* {t_v}_* \sA \rightarrow \sA, $$ where $ t_v $ denotes the inclusion of $ Y $ into the fibre of $ \tY $ over $ \spn(v) $.
\item $c$ is linear in $ V $, in the sense that $c_{v+w}(\sA) = c_v(\sA) + c_w(\sA) $.
\item If $ \sA $ deforms over $\spn(v) $ then $ c_v(\sA) = 0 $.
\end{enumerate}
\end{Proposition}
\begin{proof}
The functoriality follows from the definition of $c_v $ as a morphism of FM kernels.  The linearity follows from the fact that $c_v(\sA) $ is the product of the Atiyah class of $\sA$ with the Kodaira-Spencer map $V \rightarrow H^1(\tY, T_{\tY})$ which is linear (see page 3 of \cite{HT}). 
\end{proof}

We will also need some special properties of these obstruction maps related to products and FM kernels.  Let $ Y_1, Y_2, Y_3 $ be three varieties, all with deformations $ \tY_1, \tY_2, \tY_3 $ over the same base $ V $.  Then the pairwise products $ Y_1 \times Y_2 $ and $ Y_2 \times Y_3 $ admit deformations $\tY_1 \times \tY_2, \tY_2 \times \tY_3$ over $ V \oplus V $.  Given $ v, w \in V$, we will consider the 1-parameter deformation $ \tY_1|_{\spn(w)} \times_\k \tY_2|_{\spn(v)}$ and write $ t_{v,w} $ for the corresponding inclusion and $ c_{v,w} $ for the obstruction map (the reason for ``switching'' the order of $ v,w $ will become clear in a moment).

Let $ \sA_{12} \in D(Y_1 \times Y_2) $ and $ \sA_{23} \in D(Y_2 \times Y_3) $. Let $ v, w\in V $.  Since $ \sA_{23} * (\cdot) $ is a functor, we get a map $ \sA_{23} * \sA_{12}[-1] \xrightarrow{I c_{v,w}(\sA_{12})} \sA_{23} * \sA_{12}[1]$. 

\begin{Proposition} \label{th:propc2}
With the above notation, the following holds for all $ v \in V$.
\begin{enumerate}
\item We have an equality $ I c_{0,v}(\sA_{12}) = c_{0,v}(\sA_{12}*\sA_{23})$
\item There is distinguished triangle
$$
\sA_{23}*\sA_{12}[-1] \xrightarrow{ Ic_{v,0}(\sA_{12}) } \sA_{23}*\sA_{12}[1] \rightarrow {t_{0,v}}_* (\sA_{23}) * {t_{v,0}}_*  (\sA_{12})
$$ 
\end{enumerate}
\end{Proposition}

\begin{proof}
(i) follows from functoriality.  (ii) follows from \cite[Lemma 4.1]{ckl2}.
\end{proof}

\subsubsection{Deformations in geometric categorical $ \g $ actions} \label{se:deform}
In a geometric categorical $ \g $ action, we have a deformation  $ \tY(\l) \rightarrow \h'_\k $.  In light of  Proposition \ref{th:propc1}.(iii), condition (x) implies that $ c_{v,v}(\sE_i) = 0 $ for $ v \in \alpha_i^\perp $.  Similarly, condition (xi) implies that $ c_{v,v} (\sE_{ij}) = 0 $ for $ v \in (\alpha_i + \alpha_j)^\perp$.

Let us now study $ c_{\Lambda_i,\Lambda_i}(\sE_i) $ and in particular the map 
$$
Ic_{\Lambda_i, \Lambda_i}(\sE_i) : \sE_i * \sE_i[-1] \rightarrow \sE_i * \sE_i[1]
$$
By Proposition \ref{th:propc1}.(ii),  $c_{\Lambda_i, \Lambda_i}(\sE_i) = c_{\Lambda_i, 0}(\sE_i) + c_{0,\Lambda_i}(\sE_i)$.  

First consider $ I c_{\Lambda_i,0}( \sE_i) : \sE_i * \sE_i [-1] \rightarrow \sE_i * \sE_i [1] $.  We will examine this map on the level of cohomology.  Proposition \ref{th:propc2}.(ii) and Condition (vi) imply that 
$$\sH^*(\Cone(I c_{\Lambda_i, 0}( \sE_i))) \cong \sE_i^{(2)} [-1] \oplus \sE_i^{(2)} [2].$$

This means that on cohomology, $I c_{\Lambda_i, 0}(\sE_i)$ induces an isomorphism 
$$ \sE_i^{(2)} \cong \sH^0(\sE_i * \sE_i [-1]) \rightarrow \sH^0(\sE_i * \sE_i [1]) \cong \sE_i^{(2)}.$$

On the other hand, by Proposition \ref{th:propc2}.(i), $ Ic_{ 0,\Lambda_i}(\sE_i) = c_{0,\Lambda_i}(\sE_i * \sE_i) $.  Hence $ Ic_{0,\Lambda_i}(\sE_i)$ is given by a diagonal matrix 
$$\left( \begin{matrix} * & 0 \\ 0 & * \end{matrix} \right): \left( \begin{matrix} \sE_i^{(2)} [-2] \\ \sE_i \end{matrix} \right) \rightarrow \left( \begin{matrix} \sE_i^{(2)} \\ \sE_i^{(2)}[2] \end{matrix} \right)$$
and thus acts by 0 at the level of cohomology. 

Combining together these observations, we deduce that $ Ic_{\Lambda_i, \Lambda_i}(\sE_i) $ gives an isomorphism $ \sH^0(\sE_i * \sE_i [-1]) \rightarrow \sH^0(\sE_i * \sE_i [1])$.

We can summarize the results of this section, with the following statement.

\begin{Proposition} \label{th:deformsum}
Let $ v \in \h'_k $. 
\begin{enumerate}
\item If $ \la v, \alpha_i \ra = 0 $, then $ c_{v,v}(\sE_i) = 0 $.
\item If $ \la v, \alpha_i \ra \ne 0 $, then $ c_{v,v}(\sE_i) $ is non-degenerate in the sense that the resulting map
$$ \sH^0(Ic_{v,v}(\sE_i)) : \sE_i^{(2)} \cong \sH^0(\sE_i * \sE_i [-1]) \rightarrow \sH^0(\sE_i * \sE_i [1]) \cong \sE_i^{(2)}$$
is an isomorphism.
\end{enumerate}
\end{Proposition}

\begin{Remark}
For all practical purposes we do not actually need that $\sE_i$ deforms over all of $\alpha_i^\perp$ and $\sE_{ij}$ deforms over all of $(\alpha_i + \alpha_j)^\perp$. We just need them to deform over some one-dimensional subspaces. We use the definition above since it is satisfied in all examples we know and, in our opinion, is more aesthetically pleasing. 
\end{Remark}

\subsection{$\k^\times$-equivariance}

There is a slightly more general $\k^\times$-equivariant version of a geometric categorical $\g$ action. In this setting every variety $Y(\l)$ and deformation $\tY(\l)$ is equipped with a $\k^\times$-action. The action on $\tY(\l)$ is equivariant with respect to a weight 2 action on the base $\h'_\k$. Subsequently, each $D(Y(\l))$ becomes the derived category of $\k^\times$-equivariant coherent sheaves. This extra $\k^\times$-structure gives us another grading on our categories which we denote by $\{ \cdot \}$. More precisely, $\{k\}$ is tensoring with the line bundle $\O\{k\}$ where if $f \in \O(U)$ is a local function and $t \in \k^\times$ then viewed as a section $f' \in \O\{k\}(U)$ we have $t \cdot f' = t^{-k}(t \cdot f)$. 

The conditions on the data are essentially the same. The adjoint conditions become 
\begin{enumerate}
\item $\sE^{(r)}_i(\l)_R = \sF^{(r)}_i(\l) [r(\la \l, \alpha_i \ra + r)] \{-r(\la \l, \alpha_i \ra + r) \}$ 
\item $\sE^{(r)}_i(\l)_L = \sF^{(r)}_i(\l) [-r(\la \l, \alpha_i \ra + r)] \{r(\la \l, \alpha_i \ra + r)\}.$
\end{enumerate}
Condition (\ref{co:EEdef}) on composition of deformed $\sE$s becomes
$$\sH^*(i_{23*} \sE_i * i_{12*} \sE_i) \cong \sE_i^{(2)}[-1]\{1\} \oplus \sE_i^{(2)}[2]\{-3\}$$
while $\sE_{ij}$ is defined as the cone of $\sE_i * \sE_j [-1]\{1\} \rightarrow \sE_j * \sE_i$. All the other conditions are the same once we replace $H^\star(\p^n)$ by the doubly graded version 
$$H^\star(\p^n) := \k[-n]\{n\} \oplus \k[-n+2]\{n-2\} \oplus \dots \oplus \k[n]\{-n\}.$$

In this setup, Theorem \ref{thm:naive} shows that the functors $\E_i$ and $\F_i$ acting on the Grothendieck groups $\{K^{\k^\times}(Y(\l))\}$ gives us a representation of the quantum enveloping algebra $U_q(\g)$. 

All the results in this paper have natural $\k^\times$-equivariant analogues. However, we will not work $\k^\times$-equivariantly because keeping track of the extra $\{ \cdot \}$ shifts would make the notation hard to read. One of the reasons to even consider this $\k^\times$-equivariant setup is that it shows up naturally in various examples. The cotangent bundles of partial flag varieties considered in section \ref{sec:example} is one such example.

\subsection{Example from resolutions of Kleinian singularities} \label{se:eg}

An instructive example of geometric categorical $ \g $ action comes from the minimal resolution of a Kleinian singularity.  Let $ H $ denote a finite subgroup of $ SL_2(\C) $ and let $ \pi : Y \rightarrow \C^n/H $ be a minimal resolution.  Recall that $H $ determines a finite type simply-laced Dynkin diagram $\Gamma $ whose vertex set $I $ is in bijection with the components of the exceptional fibre $ \pi^{-1}(0) $.   

From the work of Seidel-Thomas \cite{ST}, we know that each component $ E_i $ of $ \pi^{-1}(0) $ determines a spherical object $ \sS_i = \O_{E_i}(-1) $.  These $ \sS_i $ form a type $ \Gamma $ arrangement of spherical object and thus by the work of Seidel-Thomas give an action of the braid group $ B_\g $ on $D(Y) $ (as usual, here $ \g $ is the Lie algebra associated to $ \Gamma$).  

Let us use the same data to construct a $ \C^\times $-equivariant geometric categorical $ \g $ action.  
Let $ Y(\l) $ be defined as follows. Let $ Y(0) = Y $, $ Y(\l) = \pt $ for $ \l $ a root of $ \g $, and $ Y(\l) = \emptyset $ for all other $ \l $. The action of $ \C^\times $ on $ Y $ is comes from the scaling action on $ \C^n $.  We define $ \sE_i(0): D(Y) \rightarrow D(\pt)$ using the kernel $\sS_i \in D(Y \times \pt)$ and similarly with $\sE_i(-\alpha_i)$, $\sF_i(0)$ and $\sF_i(-\alpha_i)$ (all other $ \sE_i, \sF_i $ we need to define are functors $D(\pt) \rightarrow D(\pt)$ which we take to be the identity). The deformation $\tilde{Y} $ of $ Y $ is the standard deformation (which may be constructed by thinking of $\C^n/H $ as a Slodowy slice or by deforming the polynomial defining the singularity $\C^n/H$).

Let us check condition (\ref{co:EiEj}) of the geometric categorical $ \sl_n $ action (all other conditions are immediate or follow along the same lines).  Let $ i, j$ be connected by an edge in $\Gamma $ (so $ E_i, E_j $ intersect in a point).  Then condition (\ref{co:EiEj}) states that 
$$ \sE_i(\alpha_{j}) * \sE_{j}(0) * \sE_i (-\alpha_i) \cong \sE_{j}(\alpha_i) * \sE^{(2)}_i(-\alpha_i) \oplus \sE_i^{(2)}(-\alpha_i + \alpha_{j}) * \sE_{j}(-\alpha_i).$$
Now $ Y(-\alpha_i + \alpha_{j}) = \emptyset$ while $\sE^{(2)}_i(-\alpha_i) = \O_{\Delta_{\pt}} = \sE_i(\alpha_{j})$. So we see that this is equivalent to the fact that the composition 
$$D(\pt) \xrightarrow{\sE_i(-\alpha_i)} D(Y) \xrightarrow{\sE_{j}(0)} D(\pt)$$ 
is the identity. Since the first functor is given by tensoring with the object $\O_{E_i}(-1)$ and the second functor is $\Ext^*(\O_{E_j}(-1)[-1], \cdot)$, this condition corresponds to the fact that $$\Ext^m(\O_{E_j}(-1), \O_{E_i}(-1))=0,$$  unless $m=1$ in which case it is $\C$. 

\subsection{Equivalences via geometric categorical $\sl_2$ actions}

In \cite{ckl2} we proved that a geometric categorical $\sl_2$ action induces a strong categorical $\sl_2$ action. In \cite{ckl3} we showed that a strong categorical $\sl_2$ action can be used to construct equivalences (using ideas of Chuang-Rouquier \cite{CR}). We briefly review this construction starting from a categorical $\g$ action. 

Given a geometric categorical $\g$ action one can construct for each vertex $i \in I$ a geometric categorical $\sl_2$ action. More precisely, we use as kernels $\sE^{(r)}_i$ and $\sF^{(r)}_i$ and use the one parameter deformation $\tY_i(\l)$.  Consequently by the main result of \cite{ckl2}, we obtain a strong $\sl_2$ action generated by the functors induced by the kernels $\sE_i$ and $\sF_i$. 

Consider for each $s \ge 0$ and $\la \l, \alpha_i \ra \ge 0$ the kernel
$$\sT_i^s(\l) := \sF_i^{(\la \l, \alpha_i \ra + s)} * \sE_i^{(s)}(\l)[-s] \in D(Y(\l) \times Y(\l - \la \l, \alpha_i \ra \alpha_i))$$
(if $\la \l, \alpha_i \ra \le 0$ then we consider instead the kernel $\sE_i^{(- \la \l, \alpha_i \ra + s)} * \sF_i^{(s)}$). There exists a natural map $d_i^s: \sT_i^s(\l) \rightarrow \sT_i^{s-1}(\l)$ given as the composition
\begin{eqnarray*}
\sT_i^s(\l)
&\cong& \sF_i^{(\la \l, \alpha_i \ra + s)} * \sE_i^{(s)} [-s]  \\
&\xrightarrow{\i \i}& \sF_i^{(\la \l, \alpha_i \ra + s-1)} * \sF_i [- \la \l, \alpha_i \ra - s+1] * \sE_i * \sE_i^{(s-1)} [-(s-1)][-s] \\
&\xrightarrow{I \e I}& \sF_i^{(\la \l, \alpha_i \ra + s-1)} * \sE_i^{(s-1)} [-s+1] \cong \sT_i^{s-1}(\l). 
\end{eqnarray*}

Note that $\sT_i^s(\l) = 0$ for $s \gg 0$ since we only deal with integrable representations. The main result of \cite{ckl3} is the following. 

\begin{Theorem}\label{thm:ckl3} 
$$\dots \rightarrow \sT_i^s(\l) \xrightarrow{d_i^s} \sT_i^{s-1}(\l) \xrightarrow{d_i^{s-1}} \dots \xrightarrow{d_i^1} \sT_i^0(\l)$$
is a complex of kernels which has a unique right convolution denoted $\sT_i(\l)$. Moreover, the kernel $\sT_i(\l)$ induces an equivalence of triangulated categories $D(Y(\l)) \xrightarrow{\sim} D(Y(\l - \la \l, \alpha_i \ra \alpha_i))$. 
\end{Theorem}

In the theorem above, by right convolution we mean an iterated cone starting from the right (see section \ref{sec:step3}). 

\subsection{Braid group action via geometric categorical $\g$ actions}

As noted in the previous section, each vertex $i \in I$ induces a geometric categorical $\sl_2$ action and subsequently an equivalence $\sT_i$ (or, more precisely, a series of equivalences $\sT_i(\l)$, one for each weight $\l$). The main result of this paper is to prove that these equivalences braid. 

\begin{Theorem}\label{thm:main} 
Let $ Y(\lambda), \dots $ be a geometric categorical $ \sl_2 $ action.

If $i,j \in I$ are joined by an edge then the corresponding equivalences $\sT_i$ and $\sT_j$ satisfy the braid relation $\sT_i * \sT_j * \sT_i \cong \sT_j * \sT_i * \sT_j$. If $i, j \in I$ are not joined by an edge then $\sT_i * \sT_j \cong \sT_j * \sT_i$.  Hence there is an action of the braid group of type $ \g $ on $ D(\sqcup Y(\l)) $ compatible with the action of the Weyl group on the weight lattice.   
\end{Theorem}

Let us examine this action on the level of the Grothendieck groups $ \oplus K(Y(\l)) $.  The above theorem provides us with an action of the braid group $ B_\g $ on $ \oplus K(Y(\lambda))$.  On the other hand, Theorem \ref{thm:naive} provides us with an action of $ U_q(\g)$ on $ \oplus K(Y(\l)) $.  These two structures are compatible via Lusztig's quantum Weyl group map $B_\g \rightarrow \widehat{U}_q(\g) $.  This follows from \cite{ckl3}. (In the non-equivariant case, i.e. $q=1$, then this is the same as the usual map $B_\g \rightarrow \widehat{U}(\g) $).

\begin{Example}
As a simple application of this theorem, we can consider the minimal resolution $ Y$ of the Kleinian singularity $\C^2/H$.  In section \ref{se:eg}, we explained that $ Y = Y(0) $ was the 0 weight space of a geometric categorical $ \g $ action.  Hence by Theorem \ref{thm:main}, we obtain an action of the braid group $ B_\g$ on $ D(Y)$.  As mentioned earlier, such an action was previously studied by Seidel-Thomas \cite{ST}.  A more substantial application will be given in the next section.
\end{Example}

\section{Example: cotangent bundles to flag varieties}\label{sec:example}

Before we prove Theorem \ref{thm:main} we would like to illustrate a geometric categorical $\sl_n$ action on the $\C^\times$-equivariant derived category of coherent sheaves on the cotangent bundle to partial flag varieties (Theorem \ref{thm:sl_nexample}).  We work $\C^\times$-equivariantly in order to have condition (\ref{co:trivial}) hold (if not, $\End(\O_{Y(\l)}) \cong H^0(\O_{Y(\l)}) $ will be infinite dimensional). Functors will always be considered in the derived sense (i.e. as functors between derived categories). 

\subsection{The categorical $ \g $ action}

Fix integers $ n \le N $.  We consider the variety $ Fl_n(\C^N) $ of $n$-step flags in $\C^N $.  This variety has many connected components, which are indexed by the possible dimensions of the spaces in the flags.  In particular, let 
\begin{equation*}
C(n, N) := \{ \l = (\l_1, \dots, \l_n) \in \mathbb{N}^N: \l_1 + \dots + \l_n = N \}.
\end{equation*}

For $ \l \in C(n,N)$, we can consider the variety of $n$-steps flags where the jumps are given by $ \l$:
\begin{equation*}
Fl_\l(\C^N) := \{ 0= V_0 \subset V_1 \subset \cdots \subset V_n = \C^N : \dim V_i/V_{i-1} = \l_i \}.
\end{equation*}

Let $ Y(\l) = T^\star Fl_\l(\C^N)$ (if $\l \not\in C(n,N)$ we take $Y(\l) = \emptyset$).  These will be our varieties for the geometric categorical $\sl_n $ action.  We regard each $ \l $ as a weight for $ \sl_n $ via the identification of the weight lattice of $ \sl_n $ with the quotient $ \Z^n / (1, \cdots, 1)$.  For compatibility with \cite{ckl4}, we choose the convention that the simple roots $\alpha_i$ are equal to $(0, \dots, 0,-1,1,0, \dots, 0)$ where the $-1$ is in position $i$.  

We will make use of the following description of the cotangent bundle to the partial flag varieties.  
\begin{equation*}
Y(\l) = \{(X, V): X \in \End(\C^N), V \in Fl_\l(\C^N), X V_i \subset V_{i-1} \}
\end{equation*}
This description immediately leads to the following deformations of $Y(\l) $ over $ \C^n $
\begin{equation*}
\tilde{Y}(\l) := \{(X, V, x) : X \in \End(\C^N), V \in Fl_\l(\C^N), x \in \C^n, X V_i \subset V_i, X |_{V_i/V_{i-1}} = x_i \cdot \id \}. 
\end{equation*}
In more Lie-theoretic terms, $ Fl_\l(\C^N)$ is the variety of parabolic subalgebras $\mathfrak{p}$ of $ \gl_N $ of type $ \l$.  $ Y(\l) $ is the variety of pairs $(X, \mathfrak{p}) $ where $ X \in \gl_N $, $ \mathfrak{p} $ is a parabolic subalgebra of type $\l$ and $ X $ is in the nilradical of $ \mathfrak{p} $.  Finally $ \tilde{Y}(\l) $ is the variety of triples $ (X, \mathfrak{p}, x) $ where $ X $ is in $ \mathfrak{p} $ and its image in the Levi of $\mathfrak{p} $ is the central element $ x $.

We will restrict our deformation over the locus $ \{(x_1, \dots, x_n) \in \C^n : x_n = 0 \} $ which we identify with $ \h' $, the Cartan for $ \sl_n $.

We define an action of $ \C^\times $ on $ \tilde{Y}(\l) $ by $ t \cdot (X, V, x) = (t^2 X, V, t^2 x) $.  Restricting to $ Y(\l) = T^\star Fl_\l(\C^N) $ this corresponds to a trivial action on the base and a scaling of the fibres.

To construct the kernels $ \sE^{(r)}_i $, we consider correspondences $ W_i^{r} $.  More specifically, let $ \l, i,r $ be such that $ \l \in C(n,N) $ and $ \l + r \alpha_i \in C(n,N) $ (ie $ \lambda_{i} \ge r $).  Then we define
\begin{equation*} 
W_i^{r}(\l) := \{ (X, V, V') : (X, V) \in Y(\l), (X,V') \in Y(\l + r\alpha_i), V_j = V'_j \text{ for } j \ne i, \text{ and } V'_i \subset V_i \}
\end{equation*}
From this correspondence we define the kernel 
\begin{equation*}
 \sE^{(r)}_i(\l) = \O_{W_i^{r}(\l)} \otimes \det(V_{i+1}/V_i)^{-r} \otimes \det(V'_i/V_{i-1})^r \{r(\l_i-r)\}  \in D(Y(\l) \times Y(\l+r\alpha_i) )
 \end{equation*}
where $\{\cdot\}$ denotes an equivariant shift and, abusing notation, $V_i$ denotes the vector bundle on $Y(\l)$ whose fibre over $(X,V) \in Y(\l)$ is naturally identified with $V_i$. Similarly, we define the kernel
\begin{equation*}
 \sF^{(r)}_i(\l) = \O_{W_i^{r}(\l)} \otimes \det(V'_i/V_i)^{\l_{i+1} - \l_i + r}\{ r \l_{i+1} \}  \in D(Y(\l+r \alpha_i) \times Y(\l) ).
 \end{equation*}
Note that now we regard $ W_i^{r}(\l) $ as a subvariety of $ Y(\l+r \alpha_i) \times Y(\l) $ which means that $V_i \subset V_i'$ (we will continue to use this convention). 

\begin{Theorem}\label{thm:sl_nexample}
This datum defines a geometric categorical $ \sl_n $ action on $D(T^\star Fl_n(\C^N))$. 
\end{Theorem}

By Theorem \ref{thm:naive}, this gives us a representation of $ U_q(\sl_n)$ on $$ K(T^\star Fl_n(\C^N)) = \oplus_\lambda  K(T^\star Fl_\lambda(\C^N)).$$    
Since 
$$ \dim K(T^\star Fl_\lambda(\C^N)) = \begin{pmatrix} N \\ \lambda_1 \cdots \lambda_n \end{pmatrix}  $$
and since representations of $ U_q(\sl_n) $ are determined by the dimensions of their weight spaces, we can identify this representation with $V_{\Lambda_1}^{\otimes N}$.

\subsection{The braid group action}

As a corollary of this theorem and the main result of this paper (Theorem \ref{thm:main}), we obtain the following.

\begin{Theorem}
There is an action of the braid group $ B_n $ on the derived category of coherent sheaves on $ T^\star Fl_n(\C^N) $.  This action is compatible with the action of $ S_n $ on the set of connected components $ C(n,N)$.
\end{Theorem}

In particular, if $ N = dn $ for some integer $ d $ and we choose $ \l = (d, \dots, d) $, then we obtain an action of the braid group of the derived category of coherent sheaves on the connected variety $ T^\star Fl_\l(\C^N) $.

\begin{Example}
Consider the case $n =N$.  Let $ T^\star(Fl(\C^n)) $ denote the cotangent bundle to the full flag variety. We have constructed an action of the braid group $ B_n $ on $D(T^\star(Fl(\C^n)))$.  Such an action was previously constructed by Khovanov-Thomas \cite{KT} and by Riche \cite{Ric}, Bezrukavnikov-Mirkovic-Rumynin \cite{BMR}.  Their work served as motivation for this paper.  In this case, the generators of the braid group act by spherical twists (see \cite[section 2.5]{ckl3}).  This is the simplest case of our result, since in general the equivalences which generate the braid group action are given by more complicated complexes than spherical twists.

In \cite{KT,Ric,BMR}, the braid group action is extended to an affine braid group action. The extra generators for this extended action are given by tensoring with certain line bundles. One can similarly construct such an affine braid group action on $D(T^\star(Fl_n(\C^N)))$ using line bundles. We discuss this in greater detail in \cite{ckl4} where we build on the results from this paper to construct affine braid group actions on Nakajima quiver varieties.
\end{Example}

Even though the construction of each equivalence 
$$\T_i: D(T^\star Fl_\l(\C^N)) \rightarrow D(T^\star Fl_{\l-\la \l, \alpha_i \ra \alpha_i}(\C^N))$$
via a categorical $\sl_2$ action is indirect, the kernels one obtains are fairly concrete. More precisely, the kernel $\sT_i$ which induces the functor $\T_i$ is always a sheaf supported on the variety 
\begin{eqnarray*}
Z_i(\l) := \{(X,V,V')&:& X \in \End(\C^N), V \in Fl_\l(\C^N), V' \in Fl_{\l-\la \l, \alpha_i \ra \alpha_i}(\C^N) \\
&& XV_j \subset V_{j-1}, XV'_j \subset V'_{j-1} \text{ and } V_j=V'_j \text{ if } j \ne i \}.
\end{eqnarray*}
In general $\sT_i$ is not the structure sheaf of $Z_i(\l)$ but rather some rank one Cohen-Macaulay sheaf on $Z_i(\l)$. In \cite{C} we give a concrete description of this sheaf in the Grassmannian case ($n=2$ case). A similar description of $\sT_i$ is possible in general.

In the rest of this section we prove Theorem \ref{thm:sl_nexample}.

\subsection{Proof of $\sl_2$ conditions (\ref{co:trivial}) - (\ref{co:contain})}

In \cite{ckl2} (based on the computations in \cite{ckl1}) we proved that the $\sl_2$ relations (\ref{co:trivial}) - (\ref{co:contain}) hold for cotangent bundles to Grassmannians (i.e. the case $n=2$). The same proof with virtually no changes necessary applies to prove these relations for any $n$. 

As an example, we will check the adjunction relations (\ref{co:adjoints}).  We begin by computing some canonical bundles.

\begin{Lemma}
$\omega_{T^\star Fl_\l(\C^N)} \cong \O_{T^\star Fl_\l(\C^N)} \{- 2 \sum_{i < j} \lambda_i \lambda_j \}$
\end{Lemma}
\begin{proof}
The variety $ T^\star Fl_\l(\C^N) $ is symplectic (since it is a cotangent bundle) and the symplectic form has weight $ 2 $ for the $\C^\times $ action.   Hence the $ d$th wedge power of the symplectic form gives a non-vanishing section of the canonical bundle, where $ d $ is the dimension of $ Fl_\l(\C^N) $.  Since $ d = \sum_{i < j} \lambda_i \lambda_j $, the result follows.
\end{proof}

\begin{Lemma}
We have 
$$ \omega_{W_i^{r}(\l)} \cong \det(V_{i+1}/V_i)^{-r} \det(V_i/V'_i)^{-\lambda_i + \lambda_{i+1} + r} \det(V_i'/V_{i-1})^{r} \{-2 \sum_{i < j} \lambda_i \lambda_j +2\lambda_{i+1} r \}. $$
\end{Lemma}

\begin{proof}
Let $ \mu := (\lambda_1, \dots, \lambda_{i-1}, \lambda_i - r, r, \lambda_{i+1}, \dots, \lambda_n)$ and let $ T $ denote the variety 
\begin{equation*}
T:=  \{ (X, U) : (X, U) \in Y(\mu), X U_{i+1} \subset U_{i-1}\}
\end{equation*}

We can view $ W_i^{r} $ as a subvariety of $ T $, by setting $ U_j = V_j $ for $ j < i$, $ U_i = V'_i, U_{i+1} = V_i $ and $ U_j = V_{i-1} $ for $ j > i+1 $.  Moreover we can see that $ W_i^{r} $ is carved out of $T$ as the vanishing locus of $X: U_{i+2}/U_{i+1} \rightarrow U_{i+1}/U_i \{2\}$. Thus $W_i^{r}$ is cut out of $T$ by a section of the vector bundle $(U_{i+2}/U_{i+1})^\vee \otimes U_{i+1}/U_i \{2\}$.  

Similarly, $ T $ is cut out of $ T^\star Fl_\mu(\C^N) $ by a section of $ (U_{i+1}/U_i)^\vee \otimes U_i/U_{i-1} \{2\}$. Thus 
$$\omega_{W_i^{r}(\l)} \cong \det((U_{i+2}/U_{i+1})^\vee \otimes U_{i+1}/U_i \{2\}) \otimes \det((U_{i+1}/U_i)^\vee \otimes U_i/U_{i-1} \{2\}) \otimes \omega_{ T^\star Fl_\mu(\C^N) }|_{W_i^r(\l)}.$$
Here we use that if $A \subset B$ is cut out by a section of a vector bundle $W$ then $\omega_A = \omega_B \otimes \det(W)$.  

Combining all this with our previous calculation of $ \omega_{T^\star Fl_\mu(\C^N)} $ we obtain the desired result. 
\end{proof}

We will now give the proof of the first adjunction statement.  The other proofs are similar.
\begin{Corollary}
${\sE_i^{(r)}}(\l)_R \cong \sF_i^{(r)}(\l)[r(\lambda_{i+1} - \lambda_i) +r^2]\{-r(\lambda_{i+1} - \lambda_i) -r^2\}
$
\end{Corollary}

\begin{proof}
We have 
\begin{align*}
{\sE_i^{(r)}}(\l)_R &\cong {\sE_i^{(r)}}^\vee \otimes \pi_2^* \omega_{Y(\l)} [\dim Y(\l)] \\
&\cong \O_{W_i^{r}}^\vee \otimes \det(V_{i+1}/V'_i)^r \det(V_i/V_{i-1})^{-r} \{-r(\l_i-r)\} \otimes \pi_2^* \omega_{Y(\l)} [\dim Y(\l)] \\
&\cong \omega_{W_i^{r}(\l)} \omega_{Y(\lambda + r\alpha_i) \times Y(\l)}^\vee [-\mathrm{codim} W_i^{r}(\l)]
\otimes \det(V_{i+1}/V'_i)^r \otimes \det(V_i/V_{i-1})^{-r} \\
&\qquad \otimes \pi_2^* \omega_{Y(\l)} \{-r(\l_i-r))\}  [\dim Y(\l)] \\
&\cong \O_{W_i^{r}(\l)} \otimes \det(V_{i+1}/V'_i)^{-r} \det(V'_i/V_i)^{-\lambda_i + \lambda_{i+1} + r} \det(V_i/V_{i-1})^{r} \{-2 \sum_{i < j} \lambda_i \lambda_j +2\lambda_{i+1} r \} \\
&\qquad \otimes \pi_1^* \omega_{Y(\lambda + r\alpha_i)}^\vee \det(V_{i+1}/V'_i)^r \det(V_i/V_{i-1})^{-r} \{-r(\l_i-r)\}  [\dim Y(\l)-\mathrm{codim} W_i^{r}(\l)] \\
&\cong  \O_{W_i^{r}(\l)} \otimes \det(V'_i/V_i)^{-\lambda_i + \lambda_{i+1} + r} [r(\lambda_{i+1} - \lambda_i) + r^2] \\
&\qquad \{ 2 (\sum_{i < j} \lambda_i \lambda_j + r \l_i - r \l_{i+1} - r^2) + (- 2 \sum_{i < j} \lambda_i \lambda_j + 2 \l_{i+1} r) - r(\l_i-r) \} \\
&\cong  \sF_i^{(r)}(\l)[r(\lambda_{i+1} - \lambda_i) +r^2]\{-r(\lambda_{i+1} - \lambda_i) -r^2\}
\end{align*}
where for the last isomorphism we use that $\sF^{(r)}_i(\l) = \O_{W_i^{r}(\l)} \otimes \det(V'_i/V_i)^{\l_{i+1} - \l_i + r}\{ r \l_{i+1} \}$. 
\end{proof}

\subsection{Proof of Serre relation (\ref{co:EiEj})}

Since we are in the Lie algebra $ \sl_n $, having $ i,j \in I $ joined by an edge is equivalent to $ j = i \pm 1 $.  So let us consider $ j = i + 1 $ (the case $j = i-1$ is the same). We will show that $$ \sE_i * \sE_{i+1} * \sE_i = \sE_{i}^{(2)} * \sE_{i+1} \oplus \sE_{i+1} * \sE_i^{(2)}. $$

Here is the outline of the proof. On the left hand side computing $\sE_i * \sE_{i+1}$ is straight-forward, meaning that intersections are of the expected dimension and the pushforward is one-to-one. The intersection when computing $\sE_i * \sE_{i+1} * \sE_i$ is also of the expected dimension but contains two components $A$ and $B$. Pushing forward by $\pi_{13}$ then gives us two terms (one for each component) which are equal to $\sE_{i}^{(2)} * \sE_{i+1}$ and $\sE_{i+1} * \sE_i^{(2)}$ (these are also easy to compute). 

\begin{Lemma}
$\sE_i^{(2)} * \sE_{i+1}$ and $\sE_{i+1} * \sE_i^{(2)}$ are isomorphic to 
$$\O_{W_{i^{(2)} i+1}} \otimes \det(V_{i+2}/V_{i+1})^{-1} \det(V'_{i+1}/V_i)^{-1} \det(V'_i/V_{i-1})^2\{ 2\l_i + \l_{i+1} - 5\}$$
$$\O_{W_{i+1 i^{(2)}}} \otimes\det(V_{i+2}/V_{i+1})^{-1} \det(V_{i+1}/V_i)^{-2} \det(V'_{i+1}/V'_i) \det(V'_i/V_{i-1})^2 \{2\l_i + \l_{i+1} - 3\}$$
respectively where
$$W_{i^{(2)} i+1} :=  \bigl\{(X, V, V') : V'_i \subset V_i \subset V'_{i+1} \subset V_{i+1}, \text{ and } V_j = V'_j \text{ for } j \ne i, i+1 \bigr\}$$
$$W_{i+1 i^{(2)} } :=  \bigl\{(X, V, V') : V'_i \subset V_i, V'_{i+1} \subset V_{i+1}, XV_{i+1} \subset V'_i, \text{ and } V_j = V'_j \text{ for } j \ne i, i+1 \bigr\}$$
inside $Y(\l) \times Y(\l + 2\alpha_i + \alpha_{i+1})$.
\end{Lemma}
\begin{proof}
Computing $\sE_i^{(2)} * \sE_{i+1}$ is easy since the intersection $\pi_{12}^{-1}(W_{i+1}) \cap \pi_{23}^{-1}(W_i^{2})$ is of the expected dimension and the map $\pi_{13}$ maps this intersection one-to-one onto its image 
$$\pi_{13}(\pi_{12}^{-1}(W_{i+1}) \cap \pi_{23}^{-1}(W_i^{2})) \cong W_{i^{(2)} i+1}.$$  
So neither the tensor product nor the pushforward ${\pi_{13}}_*$ have lower or higher terms. Keeping track of the line bundles gives the result. 

Computing $\sE_{i+1} * \sE_i^{(2)}$ is very similar. 
\end{proof}

\begin{Lemma}\label{lem:Eii+1}
We have 
$$\sE_i * \sE_{i+1} \cong \O_{W_{i i+1}} \otimes \det(V_{i+2}/V_{i+1})^{\vee} \det(V_i'/V_{i-1})\{\l_i + \l_{i+1} - 2\}$$
$$\sE_{i+1} * \sE_i \cong \O_{W_{i+1 i}} \otimes \det(V_{i+2}/V_i)^{\vee} \det(V'_{i+1}/V_{i-1})\{\l_i + \l_{i+1} - 1\}$$
where
$$W_{i i+1} := \bigl\{(X, V, V') : V'_i \subset V_i \subset V'_{i+1} \subset V_{i+1}, \text{ and } V_j = V'_j \text{ for } j \ne i, i+1 \}$$
$$W_{i+1 i} := \bigl\{(X, V, V') : V'_i \subset V_i, V'_{i+1} \subset V_{i+1}, XV_{i+1} \subset V'_i, \text{ and } V_j = V'_j \text{ for } j \ne i, i+1 \bigr\}$$
inside $Y(\l + \alpha_i) \times Y(\l + 2 \alpha_i + \alpha_{i+1})$ and $Y(\l) \times Y(\l + \alpha_i + \alpha_{i+1})$ respectively.
\end{Lemma}
\begin{proof}
At the level of sets 
$$W_{i i+1} \cong \pi_{13}(\pi_{12}^{-1}(W_{i+1}) \cap \pi_{23}^{-1}(W_i))$$
where the intersection is transverse and the push forward is one-to-one. Hence in computing $\sE_i * \sE_{i+1}$ neither the tensor product nor the pushforward $\pi_{13*}$ have lower or higher terms. Keeping track of the line bundles gives the result. Computing $\sE_{i+1} * \sE_i$ is very similar.
\end{proof}

Notice that the varieties $ W_{i^{(2)} i+1}, W_{i+1 i^{(2)}}, W_{i i+1} $ are all smooth.  This is because each of them is a vector bundle over a iterated Grassmannian bundle.  The fibre of these vector bundles is given by the $ X $ data and the base is given by the $ V, V' $ data.

Now we can compute $ (\sE_i * \sE_{i+1}) * \sE_i $. We find that
\begin{eqnarray*}
\pi_{12}^{-1}(W_i) \cap \pi_{23}^{-1}(W_{i i+1}) = \bigl\{(X, V, V',V'') &:& V''_i \subset V'_i \subset V_i, V'_i \subset V''_{i+1} \subset V_{i+1} = V'_{i+1}, \\
&& \text{ and } V_j = V'_j = V_j'' \text{ if } j \ne i, i+1 \bigr\}.
\end{eqnarray*}
This intersection is of the expected dimension but the push-forward under $ \pi_{13} $ is only generically one-to-one. This variety has two components which we denote by $ A $ and $ B$ which are defined by
$$A :=  \bigl\{(X, V, V',V'') : V''_i \subset V'_i \subset V_i \subset V''_{i+1} \subset V_{i+1} = V'_{i+1}, \text{ and } V_j = V'_j = V_j'' \text{ if } j \ne i, i+1 \bigr\}$$
\begin{eqnarray*}
B := \bigl\{(X, V, V',V'') &:& V''_i \subset V'_i \subset V_i, V'_i \subset V''_{i+1} \subset V_{i+1} = V'_{i+1}, XV_{i+1} \subset V''_i, \\
&& \text{ and } V_j = V'_j = V_j'' \text{ if } j \ne i, i+1 \bigr\}.
\end{eqnarray*}
The varieties $ A, B $ are smooth for the same reasons as explained above for $ W_{i^{(2)} i+1}, W_{i+1 i^{(2)}}, W_{i i+1} $.

Keeping track of the line bundles shows that $ (\sE_i * \sE_{i+1}) * \sE_i \cong {\pi_{13}}_* \big( \O_{A \cup B} \otimes \sL)$ where 
$$ \sL := \det(V_{i+2}/V_i)^{-1} \det(V'_i/V''_i) \det(V''_i/V_{i-1})^2\{2 \l_i + \l_{i+1} -3 \}.$$
Let $ E := A \cap B $.  It is a divisor inside of each of $ A, B $.  Consider the standard short exact sequence
\begin{equation*}
0 \rightarrow \O_A(-E) \oplus \O_B(-E) \rightarrow \O_{A \cup B} \rightarrow \O_E \rightarrow 0. 
\end{equation*}
Now, $ E $ is cut out of $ A $ by a section of $ \Hom(V_{i+1}/V''_{i+1}, V'_i/V''_i \{2\}) $, namely the map $X: V_{i+1}/V''_{i+1} \rightarrow V'_i/V''_i \{2\}$. Hence
$$ \O_A(-E) = \O_A \otimes (V_{i+1}/V''_{i+1}) \otimes (V'_i/V''_i)^\vee\{-2\} .$$
Similarly, $ E $ is cut out of $ B $ by a section of $ \Hom(V_i/V'_i, V_{i+1}/V''_{i+1})$, namely the natural map $V_i/V'_i \rightarrow V_{i+1}/V''_{i+1}$ induced by the inclusion $V_i \rightarrow V_{i+1}$. Hence we see that 
$$ \O_B(-E) = \O_B \otimes (V_i/V'_i) \otimes (V_{i+1}/V''_{i+1})^\vee .$$  

Putting all this together, we obtain a distinguished triangle
\begin{equation} \label{eq:bigtri}  
{\pi_{13}}_*(\O_A \otimes \sL_A) \oplus {\pi_{13}}_*(\O_B \otimes  \sL_B) \rightarrow  \sE_i * \sE_{i+1} * \sE_i \rightarrow {\pi_{13}}_*(\O_E \otimes \sL)
\end{equation}
where 
\begin{align*}
\sL_A &:= \det(V_{i+2}/V_i)^{-1} \det(V_{i+1}/V''_{i+1}) \det(V''_i/V_{i-1})^2\{2 \l_i + \l_{i+1} - 5 \}, \\ 
\sL_B &:= \det(V_i/V''_i) \det(V''_i/V_{i-1})^2 \det(V_{i+2}/V_i)^{-1} \det(V_{i+1}/V''_{i+1})^{-1}\{2 \l_i + \l_{i+1} -3 \}.
\end{align*}
Now, we note that $ \pi_{13}(A) = W_{i^{(2)} i+1} $. The map $ \pi_{13}|_A: A \rightarrow W_{i^{(2)} i+1} $ is generically one-to-one and $ \sL_A $ is pulled back from $ W_{i^{(2)} i+1}$. Since $A$ and $W_{i^{(2)} i+1}$ are both smooth we have $\pi_{13*}(\O_A) \cong \O_{ W_{i^{(2)} i+1}}$ and hence
\begin{align*}
{\pi_{13}}_*(\O_A \otimes \sL_A ) &\cong \O_{W_{i^{(2)} i+1}} \otimes \det(V_{i+2}/V_{i+1})^{-1} \det(V'_{i+1}/V_i)^{-1} \det(V'_i/V_{i-1})^2\{ 2 \l_i + \l_{i+1} - 5 \}\\
 &\cong  \sE_i^{(2)} * \sE_{i+1}. 
\end{align*}

A very similar argument shows that $ {\pi_{13}}_*(\O_B \otimes \sL_B) \cong  \sE_{i+1} *\sE_i^{(2)}. $

Finally, we see that $ \pi_{13}|_E $ is a $ \p^1 $ bundle.  Moreover $ \sL $ restricts to $ \O_{\p^1}(-1) $ on these fibres.  Hence we conclude that $  {\pi_{13}}_*(\O_E \otimes \sL) = 0 $.  So distinguished triangle (\ref{eq:bigtri}) gives us an isomorphism
$$ \sE_i^{(2)} * \sE_{i+1} \oplus  \sE_{i+1} * \sE_i^{(2)} \cong \sE_i * \sE_{i+1} * \sE_i$$
as desired.

\begin{Remark} There is an interesting similarity between the proof of the braid relation in \cite{KT} and the proof of the Serre relation above. In particular, the proof of Proposition 4.6 of \cite{KT} inspired our proof above. The geometry occuring in that proof is similar to the geometry we consider here.
\end{Remark}

Finally, the identity $\sE_i * \sE_j \cong \sE_j * \sE_i$ when $i,j \in  I$ are not joined by an edge (i.e. when $|i-j| > 1$) follows from a direct calculation of both sides (all intersections are of the expected dimension and push-forwards are one-to-one so this calculation is straight-forward). The same argument works to show that $\sF_j * \sE_i \cong \sE_i * \sF_j$ for any $i,j \in I$ (condition (\ref{co:EiFj})). 

\subsection{Existence of deformations: conditions (\ref{co:Eidef}) and (\ref{co:Eijdef})}
We now explain why condition (\ref{co:Eijdef}) holds. Since we are in the Lie algebra $ \sl_n $, having $ i,j \in I $ joined by an edge is equivalent to $ j = i \pm 1 $.  So let us consider $ j = i + 1 $ (the case $j = i-1$ is the same).  We must show that the sheaf $ \sE_{i i+1} = \Cone(T_{i i+1}) $ deforms over the subspace $(\alpha_i + \alpha_{i+1})^\perp$. 

Here is an outline of the argument. Recall that $\sE_{ii+1} \cong \Cone(\sE_i * \sE_{i+1}[-1] \xrightarrow{T_{ii+1}} \sE_{i+1} * \sE_i)$. Now $\sE_i * \sE_{i+1}$ is a line bundle supported on $W_{ii+1}$ and $\sE_{i+1} * \sE_i$ a line bundle supported on $W_{i+1i}$. We show that the connecting map $T_{ii+1}$ must be (up to tensoring by a line bundle) the connecting map in the standard triangle
\begin{equation*}  
\O_{W_{i+1 i}}(-D) \rightarrow \O_{U_{i i+1}} \rightarrow \O_{W_{i i+1}}
\end{equation*}
where $U_{i i+1} = W_{i+1 i} \cup W_{i i+1}$ and $D := W_{i+1 i} \cap W_{i i+1}$. Thus we identify $\sE_{ii+1}$ with a line bundle supported on $U_{i i+1}$ and then write down an explicit deformation of it. 

\begin{Lemma}
We have 
$$\sE_{i i+1} \cong \O_{U_{i i+1}} \otimes \det(V_{i+2}/V'_i)^{\vee} \det(V_{i+1}/V_{i-1}) \{ \l_i +  \l_{i+1} -1 \}$$
where
$$U_{i i+1} := \{(X, V, V') : V'_i \subset V_i, V'_{i+1} \subset V_{i+1}, V_j = V'_j \text{ if } j \ne i, i+1 \} \subset Y(\l) \times Y(\l+\alpha_i+\alpha_j).$$
\end{Lemma}
\begin{proof}
Recall that by Lemma \ref{lem:Eii+1} the objects $\sE_i * \sE_{i+1}$ and $\sE_{i+1} * \sE_i$ in $D(Y(\l) \times Y(\l + \alpha_i + \alpha_{i+1})) $ are line bundles supported on smooth varieties $W_{i i+1}$ and $W_{i+1 i}$. Notice that the difference between the varieties $ W_{i i+1} $ and $ W_{i+1 i} $ is that in the former we demand that  $V_i \subset V'_{i+1} $, while in the latter we demand that $X V_{i+1} \subset V'_i$. 

We claim that $ U_{i i+1} = W_{i i+1} \cup W_{i+1 i} $.  To see this, let $ (X, V, V') \in U_{i i+1} $.  Then $ V'_i \subset V_i $ and $ V'_i \subset V'_{i+1} $.  Since $ \dim V_i/V'_i = 1$, this implies that either $ V_i \subset V'_{i+1} $ or $ V'_{i+1} \cap V_i = V'_i $.  If $ V_i \subset V'_{i+1} $, then $ (X, V, V') \in W_{i i+1} $.  On the other hand, if $ V'_{i+1} \cap V_i = V'_i $, then the conditions $ X V_{i+1} \subset V_i $ and $ X V_{i+2} \subset V'_{i+1} $, force $ X V_{i+1} \subset V_i \cap V'_{i+1} = V'_i $ and so we see that $(X, V, V') \in W_{i+1 i}$.  

Thus $ U_{i i+1} = W_{i i+1} \cup W_{i+1 i} $ and these are the two irreducible components of $ U_{i i+1} $.  These two components intersect in a divisor $D$ and gives us a short exact sequence of sheaves
\begin{equation*}  
0 \rightarrow \O_{W_{i+1 i}}(-D) \rightarrow \O_{U_{i i+1}} \rightarrow \O_{W_{i i+1}} \rightarrow 0. 
\end{equation*}

Using the fact that $ D $ is cut out of $ W_{i+1 i} $ by a section of $ \Hom(V_i/V'_i, V_{i+1}/V'_{i+1})$, we see that 
$$ \O_{W_{i+1 i}}(-D) \cong  \O_{W_{i+1 i}} \otimes (V_{i+1}/V'_{i+1}) \otimes (V_i/V'_i)^\vee. $$  
Substituting this into the previous exact sequence and rotating, we obtain a distinguished triangle
\begin{equation*}
\O_{W_{i i+1}}[-1] \rightarrow \O_{W_{i+1 i}} \otimes (V_{i+1}/V'_{i+1})  (V_i/V'_i)^\vee \rightarrow \O_{U_{i i+1}}.
\end{equation*}
Tensoring with the line bundle $\det(V_{i+2}/V_{i+1})^\vee \det(V_i'/V_{i-1}) \{\l_i + \l_{i+1} - 1\}$, we obtain the distinguished triangle
\begin{equation*}
\sE_i * \sE_{i+1}[-1]\{1\} \rightarrow \sE_{i+1} * \sE_i \rightarrow \O_{U_{i i+1}} \otimes \det(V_{i+2}/V_{i+1})^{\vee} \det(V_i'/V_{i-1}) \{\l_i + \l_{i+1} - 1 \}
\end{equation*}

Moreover the first map in this distinguished triangle is non-zero. Since $T_{i i+1}$ is the unique such map (up to multiple) the first map must equal $T_{i i+1}$ up to multiple. The result follows.
\end{proof}

Now that we have identified $\sE_{i i+1}$ more explicitly we can write down a deformation $\tsE_{i i+1}$ over 
$$B := (\alpha_i + \alpha_{i+1})^\perp = (0, \dots 0, -1, 0, 1, 0, \dots, 0)^\perp = \{(x_1, \dots, x_{n-1}) : x_i = x_{i+2} \}.$$
Define the variety 
\begin{align*}
\tilde{U}_{i i+1} := &\{ (X, V, V', x) : (X,V,x) \in \tilde{Y}(\l), (X, V', x) \in \tilde{Y}(\l + \alpha_i + \alpha_{i+1}), x \in (\alpha_i+\alpha_{i+1})^\perp,\\ 
&V'_i \subset V_i, V'_{i+1} \subset V_{i+1}, V_j = V'_j \text{ for } j \ne i, i+1 \}
\end{align*}
and consider 
$$\tilde{\sE}_{i i+1} := \O_{\tilde{U}_{i i+1}} \otimes  \det(V_{i+2}/V'_i)^\vee \det(V_{i+1}/V_{i-1}) \{\l_i + \l_{i+1} -1\} \in D(\tilde{Y}(\l)|_{B} \times_{B} \tilde{Y}(\l + \alpha_i + \alpha_{i+1})|_{B}).$$

\begin{Proposition} 
We have $j^* \tilde{\sE}_{i i+1} = \sE_{i i+1}$ where $j$ is the inclusion of the central fibre $ Y(\l) \times Y(\l + \alpha_i + \alpha_{i+1}) $ into $\tY(\l)|_B \times_{B} \tY(\l + \alpha_i + \alpha_{i+1})|_B$. 
\end{Proposition}
\begin{proof}
Since the line bundles on both sides agree, it suffices to show that $ j^*\O_{\tilde{U}_{i i+1}} = \O_{U_{i i+1}}$. To do this it suffices to show that $\tilde{U}_{i i+1}$ is an irreducible variety of dimension $ \dim U_{i i+1} + \dim B $ and that the scheme theoretic central fibre of $ \tilde{U}_{i i+1} \rightarrow B $ is reduced.  To do this, we will pass to local coordinates.  To simplify our task of finding local coordinates we will use an idea of Riche \cite{Ric} and pass to a subvariety.

Note that $ \tilde{U}_{i i+1} $ has an action of the group $SL_N $.  Let $ Z $ denote the subvariety of $ \tilde{U}_{i i+1} $ consisting of those points $ (X, V, V', x) $ where 
\begin{equation*}
0 \subset V_1 \subset \cdots \subset V_{i-1} \subset V'_i \subset V'_{i+1} \subset V_{i+1} \subset V_{i+2} \subset \cdots \subset V_{n-1} \subset \C^N
\end{equation*}
is the standard partial flag.  This means that $ V_1 = \spn(e_1, \dots, e_{\lambda_1}) $, etc.  The variety $ Z $ has an action of the parabolic subgroup $ P \subset SL_N $ which fixes the flag above.  Note that as before, we have a map $ Z \rightarrow B $.

The variety $ \tilde{U}_{i i+1} $ is obtained from $ Z $ by associated bundle construction, $ \tilde{U}_{i i+1} \cong Z \times_{P} SL_N $.  Moreover, this is actually an isomorphism as varieties over $ B$.  Hence it suffices to prove that $ Z $ is irreducible of expected dimension and the central fibre over $ B$ is reduced.

We will show that $ Z $ can be covered by open affine varieties $ A $ such that $ A $ can be embedded into $ \C^{n-2+r} $ (for some $r$) with the map to $ B $ given by projection onto the first $ n-2 $ coordinates.  Let us write the coordinates on $ \C^{n-2+r} $ as $ x_1, \dots,x_{i+1}, x_{i+3}, \dots,  x_{n-1}, z_1, \dots, z_r $.  We will show that under this embedding $ A $ is given by the single equation $ x_i - x_{i+1} = z_1 z_2 $.  This proves the desired facts concerning the central fibre of $ A $ and hence also for $ \tilde{U}_{i i+1} $.

To find this open affine variety $ A $, note that $ Z $ has a smooth surjective affine map $ Z \rightarrow \p(V_{i+1}/V'_i) $, taking $ (X, V, V',x) $ to $ V_i $.  Pick $ k, l $ such that 
$$ V'_{i+1} = V'_i \oplus \spn(e_k, \dots, e_{l-1}) \text{ and } V_{i+1} = V'_i \oplus \spn(e_k, \dots, e_l). $$ 
In $\p(V_{i+1}/V'_i)$ there is an open affine subspace consisting of those $ V_i $ which are of the form
\begin{equation} \label{eq:videf}
V_i = V'_i \oplus \langle e_k + c_{k+1} e_{k+1} + \cdots + c_l e_l \rangle
\end{equation}
We let $ A $ denote the preimage of this affine subspace in $ Z $.  So a point in $ A $ is described by $ (c_{k+1}, \dots, c_l) $ and $ (x_1, \dots, x_{n-1}) $ and the matrix $ X$.  We will now describe equations for $ A $.  To do this, let us introduce the variety 
$$ T = \{(X, V_i, x_1, \dots, x_{n-1}) : V_i \text{ as above in (\ref{eq:videf})}, (X-x_j) V'_j \subset V'_{j-1} \text{ for all $j$}, \text{ and } x_i =x_{i+2}  \}.
$$
Hence $ A$ is the closed subvariety of $ T$ defined by the equations $ (X-x_i)V_i \subset V_{i-1} $ and $ (X-x_{i+1})V_{i+1} \subset V_i $.  Since $ \dim(V_i/V_i') = 1 $ and $ \dim (V_{i+1}/V'_{i+1}) = 1 $, we see that these equations are equivalent to
\begin{equation*}
(X-x_i)(e_k + c_{k+1}e_{k+1} + \dots + c_le_l) \in V_{i-1} \quad (X - x_{i+1})(e_l) \in V_i.
\end{equation*}

Now, if $ (X, V_i, x_1, \dots, x_{n-1}) \in T $, then $ X $ is upper triangular where the diagonal is broken up into blocks corresponding to the $ V'_j $ with each block a diagonal matrix with $ x_j $ on the diagonal.  Hence $ T $ is an affine space with coordinates given by the entries in the matrices in the blocks above the diagonals along with $(c_{k+1}, \dots, c_l) $ and $ (x_1, \dots, x_{n-1})$.  Now let $ (X, V_i,x_1, \dots, x_{ n-1}) \in T$, and let us consider the square diagonal submatrix of $ X $ containing matrix coefficients for the basis elements $ e_{k-r}, \dots, e_l $, where $ r = \lambda_i - 1 $.  This square submatrix has the form
\begin{equation*}
\begin{bmatrix}
x_iI_r &w_k &w_{k+1} &&\dots& w_l\\
& x_{i+1} & & &&a_k\\
&&\ddots  & && a_{k+1}\\
&&& \ddots && \vdots\\
&&&& x_{i+1} & a_{l-1}\\
&&&&& x_i \\
\end{bmatrix}
\end{equation*}
where $I_r$ is the $r \times r$ identity matrix. Here we mean that 
\begin{align*}
X e_j &= x_{i+1}e_j + w_j + \dots \quad \quad \text{ for $ k \le j \le l-1 $}\\
X e_l &= x_i e_l + a_{l-1} e_{l-1} + \dots + a_k e_k + w_l + \dots 
\end{align*}
where $ w_j \in \spn(e_{k-r}, \dots, e_{k-1}) $ (so $ w_j $ is a column matrix of height $ r $) and $ \dots $ denotes terms in $ V_{i-1}$. Note also that $e_l \in V'_{i+2} $ so the $(l,l)$ matrix entry is $ x_i = x_{i+2} $.

Now $ (X, V_i,x_1, \dots, x_{n-1}) $ lies in $ A $ if and only if $ (X-x_i)(e_k + c_{k+1}e_{k+1} + \dots + c_le_l) \in V_{i-1} $ and $ (X - x_{i+1})(e_l) \in V_i$.  These conditions translate into the equations
\begin{align*}
a_{k+1} &= a_k c_{k+1} \\
&\vdots\\
a_{l-1} &= a_k c_{l-1} \\
x_i - x_{i+1} &= a_k c_l \\
w_k &= -(c_{k+1} w_{k+1} + \cdots + c_l w_l) 
\end{align*}
Hence we can embed $ A $ into the subaffine space given by the $x_1, \dots, x_{i+1},x_{i+3}, \dots,x_{n-1}$,$a_k$, the $c_j $, the entries in $w_j$ ($ j \ne k$), and all the other free matrix entries.  Inside of this affine space, $ A $ will be defined by the single equation $ x_i - x_{i+1} = a_k c_l$ as desired.
\end{proof}

\begin{Remark} It is interesting to notice that neither $W_{ii+1}$ nor $W_{i+1i}$ deform over $B$ but that $U_{ii+1} = W_{ii+1} \cup W_{i+1i}$ does deform. In fact $W_{ii+1}$ and $W_{i+1i}$ only deform over $\alpha_i^\perp \cap \alpha_{i+1}^\perp$. 
\end{Remark}

The proof that $\sE_i$ deforms over $\alpha_i^\perp$ (condition (\ref{co:Eidef})) is the same but easier since we already have an explicit description of $\sE_i$ as a line bundle supported on $W_i$.

\section{Preliminaries}\label{sec:preliminaries}

In this section we fix some further notation and prove various technical results about compositions of functors $\sE$ and $\sF$ and about spaces of maps (natural transformations) between them. The reader can choose to skim this section on a first reading, using it as a reference. 

\subsection{Some general notions} \label{se:idem}

\subsubsection{Idempotent completeness}
Let ${\mathcal C}$ be a graded additive category over $\k$ which is idempotent complete. Graded means that ${\mathcal C}$ has a shift functor $[1] $ which is an equivalence. Idempotent complete means that if $e \in \End(A)$ where $e^2=e$ then $A \cong A_1 \oplus A_2$ where $e$ acts by the identity on $A_1$ and by zero on $A_2$. Notice that the derived category of coherent sheaves on any variety is idempotent complete. This is because the derived category of any abelian category is idempotent complete (see, for instance, Corollary 2.10 of \cite{BS}). So all the categories we work with are idempotent complete.

Suppose that (each graded piece of) the space of homs between two objects is finite dimensional (by condition (\ref{co:trivial}) this is true in our setup). Then every object in ${\mathcal C}$ has a unique, up to isomorphism, direct sum decomposition into indecomposables (see section 2.2 of \cite{Rin}). Assume, moreover, the following fact:
$$\text{for any non-zero object } A \in {\mathcal C} \text{ we have } A \cong A [ k ] \Rightarrow k=0.$$

Then if $A,B,C \in {\mathcal C}$, we have the 
following cancellation laws:
\begin{eqnarray}\label{eq:A}A \oplus B \cong A \oplus C \Rightarrow B \cong C
\end{eqnarray}
\begin{eqnarray}\label{eq:B}A \otimes_\k V \cong B \otimes_\k V \Rightarrow A \cong B
\end{eqnarray}
where $V$ is a graded $\k$ vector space. The first law above follows by uniqueness of direct sum decomposition. To see the second law, decompose $A$ and $ B $ into indecomposables as
\begin{equation*}
A \cong \bigoplus_{i,j} X_i^{a_{ij}} [ j ] \quad B \cong \bigoplus_{i,j} X_i^{b_{ij}} [ j]
\end{equation*}
where $ X_i $ are indecomposable, $a_{ij}, b_{ij} \in \mathbb N$ and $X_i \cong X_{i'} [ j ]$ implies that $i = i'$.  We must show that $a_{ij} = b_{ij}$.   Now fix $i$ and consider just the summands $X_i [j ]$.  By the uniqueness of the direct sum decomposition, we get 
$$\oplus_{j} X_i^{a_{ij}} [ j ] \otimes_\k V \cong \oplus_{j} X_i^{b_{ij}} [ j ]\otimes_\k V$$
Now consider the Poincar\'e polynomials 
$$A(t) := \sum_j a_{ij} t^j, B(t) := \sum_j b_{ij} t^j \text{ and } V(t) := \sum_j \dim(V_j) t^j$$
where $V_j$ is the $j$th graded piece of $V$. Then since $A_i [ j ] \not\cong A_i [ j' ]$ for $j \ne j'$ it follows that $A(t)V(t) = B(t)V(t)$ which implies $A(t)=B(t)$ and we are done. 

\subsubsection{Bricks and ranks.} A brick is an indecomposable object $A$ in $ \mathcal{C} $ such that $\End(A) = \k \cdot \id$.  Suppose that $ A $ is a brick and that $ X, Y $ are arbitrary objects of $ \mathcal{C} $.  Let $ f :X \rightarrow Y $ be a morphism.  $f$ gives rise to a bilinear pairing $ \Hom(A, X) \times \Hom(Y, A) \rightarrow \Hom(A,A) = \k $.  We define the $ A$-rank of $ f $ to be the rank of this bilinear pairing.  

We may also define $ A$-rank as follows.  Choose (non-canonical) direct sum decompositions $X = A\otimes V \oplus B$ and $Y = A\otimes V' \oplus B'$ where $V, V' $ are $ \k $ vector spaces and $B, B'$ do not contain $A$ as a direct summand.  Then one of the matrix coefficients of $ f $ is a map $ A \otimes V \rightarrow A \otimes V' $, which (since $ A $ is a brick) is equivalent to a linear map $ V \rightarrow V' $.  The $A$-rank of $ f $ equals the rank of this linear map.  If the $A$-rank of $ f$ is $ k$, then we will say that ``$f $ gives an isomorphism on $ k $ summands isomorphic to $A $''.

Note that the notion of brick makes no reference to the shift functor in $ \mathcal C $.  On the other hand, if $ A $ is a brick, then $ A[n] $ is a brick for all $ n $.  Moreover, if $ f : X \rightarrow Y $ is a morphism, we will say that ``$f $ gives an isomorphism on $ k $ summands of the form $A[\cdot]$'' if the sum (over all $ n $) of the $ A[n]$-ranks of $ f $ is $ k $.

\subsubsection{Gaussian elimination}
Finally, we will repeatedly use the following cancellation Lemma which Bar-Natan \cite{BN} calls ``Gaussian elimination''. 
\begin{Lemma}\label{lem:cancel}
Let $ X, Y, Z, W$ be four objects in a triangulated category.  Let $f = \left( \begin{matrix} A & B \\ C & D \end{matrix} \right) : X \oplus Y \rightarrow Z \oplus W $ be a morphism.  If $D$ is an isomorphism, then $\Cone(f) \cong \Cone(A - BD^{-1}C: X \rightarrow Z)$.
\end{Lemma}
\begin{proof}
This is essentially Lemma 4.2 from \cite{BN} (or Lemma 5.25 from \cite{ck2}). 
\end{proof}

\subsection{Some basic $\sl_2$ relations}

We begin by reviewing some of the relations which follow from the definition of a geometric categorical $\sl_2$ action. These results are strictly about $\sl_2$ actions and so they all follow from \cite{ckl2}. 

\begin{Proposition}\label{prop:EXE}
We have the direct sum decomposition 
$$\sE_i * \sE_i^{(r)} \cong \sE_i^{(r+1)} \otimes_\k H^\star(\p^r) \cong \sE_i^{(r)} * \sE_i.$$
More generally, we have 
$$\sE_i^{(r_1)} * \sE_i^{(r_2)} \cong \sE_i^{(r_1+r_2)} \otimes_\k H^\star(\bG(r_1,r_1+r_2))$$
where $\bG(r_1,r_1+r_2)$ denotes the Grassmannian of $r_1$-planes in $\C^{r_1+r_2}$. 
\end{Proposition}
\begin{proof}
This follows by Proposition 4.2 and Corollary 4.7 of \cite{ckl2}. 
\end{proof}

\begin{Proposition}\label{prop:FXE}
We have the direct sum decompositions:
$$\sF_i(\l) * \sE_i(\l) \cong \sE_i(\l-\alpha_i) * \sF_i(\l-\alpha_i) \oplus \O_\Delta \otimes_\k H^\star(\p^{- \la \l, \alpha_i \ra - 1}) \text{ if } \la \l, \alpha_i \ra \le 0$$
$$\sE_i(\l-\alpha_i) * \sF_i(\l-\alpha_i) \cong \sF_i(\l) * \sE_i(\l) \oplus \O_\Delta \otimes_\k H^\star(\p^{\la \l, \alpha_i \ra - 1}) \text{ if } \la \l, \alpha_i \ra \ge 0.$$
\end{Proposition}
\begin{proof}
This follows by Proposition 4.5 of \cite{ckl2}. 
\end{proof}

\begin{Corollary}\label{cor:iso2} We have
$$\sE_i^{(b)} * \sF_i^{(a)} \cong \bigoplus_{j \ge 0} \sF_i^{(a-j)} * \sE_i^{(b-j)}(\l) \otimes_\k H^\star(\bG(j,\la \l, \alpha_i \ra -a+b)) \text{ if } \la \l, \alpha_i \ra -a+b \ge 0$$
$$\sF_i^{(b)} * \sE_i^{(a)}(\l) \cong \bigoplus_{j \ge 0} \sE_i^{(a-j)} * \sF_i^{(b-j)} \otimes_\k H^\star(\bG(j, - \la \l, \alpha_i \ra + a-b)) \text{ if } \la \l, \alpha_i \ra -a+b \le 0$$
where, by convention, $\sE_i^{(l)} = 0 = \sF_i^{(l)}$ if $l < 0$. Moreover, if $i \ne j$ then 
$$\sF_i^{(b)} * \sE_j^{(a)} \cong \sE_j^{(a)} * \sF_i^{(b)}.$$
\end{Corollary}
\begin{proof}
This first statement is a formal consequence of Proposition \ref{prop:FXE} and the cancellation relations. See \cite{ckl3} Lemma 4.2 for a sketch of the proof. The second commutation relation follows by cancellation from Proposition \ref{prop:EXE} and by repeatedly applying the fact that $\sF_i * \sE_j \cong \sE_j * \sF_i$ if $i \ne j$. 
\end{proof}

\subsection{Spaces of maps}

Next we have some results about maps between various combinations of $\sE$s.

\begin{Lemma}\label{lem:Eij} If $i,j \in I$ are joined by an edge then 
\begin{eqnarray}\label{eq:1}
\Ext^k(\sE^{(b)}_i * \sE^{(a)}_j, \sE^{(a)}_j * \sE^{(b)}_i) \cong
\left\{\begin{array}{ll}
0 & \text{if $k < ab$} \\
\k & \text{if $k = ab$}
\end{array}
\right.
\end{eqnarray}
while
\begin{eqnarray}\label{eq:2}
\Ext^k(\sE^{(b)}_i * \sE^{(a)}_j, \sE^{(b)}_i * \sE^{(a)}_j) \cong
\left\{\begin{array}{ll}
0 & \text{if $k < 0$} \\
\k \cdot \id & \text{if $k = 0$}
\end{array}
\right.
\end{eqnarray}
for any $a,b \ge 0$ (here we are assuming that $\sE_i^{(b)} * \sE_j^{(a)} \ne 0$). Thus $ \sE_i^{(b)} * \sE_j^{(a)} $  is a brick. The same results hold if we replace all $\sE$s by $\sF$s. 
\end{Lemma}
We will denote the unique map (up to non-zero multiple) in (\ref{eq:1}) when $k=ab$ by 
$$T_{ij}^{(b)(a)}: \sE^{(b)}_i * \sE^{(a)}_j[-ab] \rightarrow \sE^{(a)}_j * \sE^{(b)}_i.$$
When $a=b=1$ we omit the superscripts. 
\begin{proof}
The proof is by (decreasing) induction on $\la \l, \alpha_j \ra$ and also $a,b$. The base case being $a=0$ or $b=0$ and follows by Lemma 4.9 of \cite{ckl2}.  

Using adjunction and (\ref{co:EiFj}), we have 
\begin{eqnarray*}
& & \Ext^k(\sE^{(b)}_i(\l+ a \alpha_j) * \sE^{(a)}_j(\l), \sE^{(a)}_j(\l+ b \alpha_i) * \sE^{(b)}_i(\l)) \\
&\cong& \Ext^k(\sF^{(a)}_j(\l+ b \alpha_i) [-a(\la \l+ b\alpha_i, \alpha_j \ra + a)] * \sE^{(b)}_i(\l+ a \alpha_j) * \sE^{(a)}_j(\l), \sE^{(b)}_i(\l)) \\
&\cong& \Ext^k(\sE^{(b)}_i(\l) * \sF^{(a)}_j(\l) * \sE^{(a)}_j(\l) [-a(\la \l, \alpha_j \ra + a-b)], \sE^{(b)}_i(\l))
\end{eqnarray*}
Now, let us suppose $\la \l, \alpha_j \ra \le 0$ (if not then we rewrite the equation above by moving the left $\sE_j^{(a)}$ to the right hand side using adjunction and proceeding in the same way). Then by Corollary \ref{cor:iso2}, we have
$$\sF^{(a)}_j(\l) * \sE^{(a)}_j(\l) \cong \bigoplus_{s=0}^{s=a} \sE^{(a-s)}_j * \sF^{(a-s)}_j \otimes_\k H^\star(\bG(s, - \la \l, \alpha_j \ra))$$
so we need to understand 
$$\Ext^k(\sE_i^{(b)} * \sE_j^{(a-s)} * \sF_j^{(a-s)}(\l-(a-s)\alpha_j) \otimes_\k H^\star(\bG(s, - \la \l, \alpha_j \ra)), \sE_i^{(b)} [a(\la \l, \alpha_j \ra + a-b)]).$$
But $\sF_j^{(a-s)}(\l-(a-s)\alpha_j)_L \cong \sE_j^{(a-s)}(\l-(a-s)\alpha_j)[(a-s)(\la \l - (a-s)\alpha_j, \alpha_j \ra + a - s)]$ so this equals 
$$\Ext^k(\sE_i^{(b)} * \sE_j^{(a-s)} \otimes_\k H^\star(\bG(s, - \la \l, \alpha_j \ra)), \sE_i^{(b)} * \sE_j^{(a-s)}(\l-(a-s)\alpha_j) [(2a-s)(\la \l, \alpha_j \ra + s) -ab]).$$
Now $H^\star(\bG(s, - \la \l, \alpha_j \ra))$ is supported in degrees $* \le -s(\la \l, \alpha_j \ra+s)$. So we get summands of the form
$$\Ext^k(\sE_i^{(b)} * \sE_j^{(a-s)}, \sE_i^{(b)} * \sE_j^{(a-s)}(\l-(a-s)\alpha_j) [2(a-s)(\la \l, \alpha_j \ra + s) -ab-*])$$
where $* \ge 0$. If $k < ab$ then $2(a-s)(\la \l, \alpha_j \ra + s) -ab-*+k < 0$ so we get zero (here we use that 
$$\Ext^{<0}(\sE_i^{(b)} * \sE_j^{(a-s)}, \sE_i^{(b)} * \sE_j^{(a-s)}(\l-(a-s)\alpha_j)) = 0$$
by the induction hypothesis). If $k = ab$ then $2(a-s)(\la \l, \alpha_j \ra + s)-*$ is non-negative precisely when $*=0$ and $s=a$ and we get only one such summand since $H^\star(\bG(s, - \la \l, \alpha_j \ra))$ is one-dimensional in top degree. 

This completes half the induction argument (i.e. relation (\ref{eq:2}) implies (\ref{eq:1})). To prove the other half we repeat the analogous argument with 
$$\Ext^k(\sE_i^{(b)} * \sE_j^{(a)}, \sE_i^{(b)} * \sE_j^{(a)})$$
to show that relation (\ref{eq:2}) holds assuming (\ref{eq:1}). 

Notice that to ensure the induction terminates we need the assumption that the action is integrable. The corresponding result for $\sF$s follows by taking adjoints. 
\end{proof}

The following result shows that $ \sE_i^{(b)} * \sE_j^{(a)} $ are bricks if $ i $ and $ j $ are not connected.

\begin{Corollary}\label{cor:Eij} If $i, j \in I$ are not joined by an edge then 
\begin{eqnarray}\label{eq:3}
\Ext^k(\sE^{(b)}_i * \sE^{(a)}_j, \sE^{(b)}_i * \sE^{(a)}_j) \cong
\left\{\begin{array}{ll}
0 & \mbox{if $k < 0$} \\
\k \cdot \id & \mbox{if $k = 0$}
\end{array}
\right.
\end{eqnarray}
and similarly if we replace all the $\sE$s by $\sF$s (here we are assuming that $\sE_i^{(b)} * \sE_j^{(a)} \ne 0$). 
\end{Corollary}
\begin{proof}
The proof is precisely the induction from Lemma \ref{lem:Eij}. The main difference is that in the computation we replace $\la \alpha_i, \alpha_j \ra = -1$ by $\la \alpha_i, \alpha_j \ra = 0$. Also, the induction has only one part since now $\sE^{(b)}_i$ and $\sE^{(a)}_j$ commute (because $\sE_i$ and $\sE_j$ commute). 
\end{proof}

\subsection{Some basic $\sl_3$ relations}

We first generalize the relation $\sE_i * \sE_j * \sE_i \cong \sE_i^{(2)} * \sE_j \oplus \sE_j * \sE_i^{(2)}$ when $i,j \in I$ are joined by an edge. 

\begin{Proposition}\label{prop:E1E2E1} If $i, j \in I$ are joined by an edge then 
$$\sE_i^{(a)} * \sE_j * \sE_i \cong \sE_i^{(a+1)} * \sE_j \otimes_\k H^\star(\p^{a-1}) \oplus \sE_j * \sE_i^{(a+1)}$$
and similarly 
$$\sE_i * \sE_j * \sE_i^{(a)} \cong \sE_j * \sE_i^{(a+1)} \otimes_\k H^\star(\p^{a-1}) \oplus \sE_i^{(a+1)} * \sE_j.$$
\end{Proposition}
\begin{proof}
We prove the first relation by induction on $a$ (the second relation follows similarly). The base case is $a=1$ which is precisely one of the conditions of having a geometric categorical $\g$ action. 

Applying $\sE_i^{(a)}$ to the left of the relation for $a=1$ we get
$$\sE_i^{(a)} * \sE_i * \sE_j * \sE_i \cong \sE_i^{(a)} * \sE_j * \sE_i^{(2)} \oplus \sE_i^{(a)} * \sE_i^{(2)} * \sE_j.$$
Tensoring both sides with $H^\star(\p^1)$ and using that $\sE_i^2 = \sE_i^{(2)} \otimes_\k H^\star(\p^1)$ we get
\begin{eqnarray*}
\sE_i^{(a)} * \sE_i * \sE_j * \sE_i \otimes_\k H^\star(\p^1) 
&\cong& \sE_i^{(a)} * \sE_j * \sE_i^2 \oplus \sE_i^{(a)} * \sE_i^2 * \sE_j \\
&\cong& \left( \sE_i^{(a+1)} * \sE_j * \sE_i \otimes_\k H^\star(\p^{a-1}) \oplus \sE_j * \sE_i^{(a+1)} * \sE_i \right) \oplus \sE_i^{(a)} * \sE_i^2 * \sE_j
\end{eqnarray*}
where the second isomorphism follows by induction. Thus we get
$$\sE_i^{(a+1)} * \sE_j * \sE_i \otimes_\k H^\star(\p^1 \times \p^a) \cong \sE_i^{(a+1)} * \sE_j * \sE_i \otimes_\k H^\star(\p^{a-1}) \oplus \sE_j * \sE_i^{(a+1)} * \sE_i \oplus \sE_i^{(a)} * \sE_i^2 * \sE_j.$$
By cancellation law (\ref{eq:A}), we obtain
\begin{eqnarray*}
\sE_i^{(a+1)} * \sE_j * \sE_i \otimes_\k H^\star(\p^{a+1}) 
&\cong& \sE_j * \sE_i^{(a+1)} * \sE_i \oplus \sE_i^{(a)} * \sE_i^2 * \sE_j \\
&\cong& \sE_j * \sE_i^{(a+2)} \otimes_\k H^\star(\p^{a+1}) \oplus \sE_i^{(a+2)} * \sE_j \otimes_\k H^\star(\p^a \times \p^{a+1}).
\end{eqnarray*}
Using cancellation law (\ref{eq:B}) we get 
$$\sE_i^{(a+1)} * \sE_j * \sE_i \cong \sE_j * \sE_i^{(a+2)} \oplus \sE_i^{(a+2)} * \sE_j \otimes_\k H^\star(\p^a)$$
and the induction step is complete. 
\end{proof}

\begin{Corollary}\label{cor:E1E2E1} If $i, j \in I$ are joined by an edge then 
$$\sE_i^{(a)} * \sE_j * \sE_i^{(b)} \cong \sE_i^{(a+b)} * \sE_j \otimes_\k H^\star(\bG(b,a+b-1)) \oplus \sE_j * \sE_i^{(a+b)} \otimes_\k H^\star(\bG(a,a+b-1)).$$
\end{Corollary}
\begin{proof}
The proof is by induction on $a$. To compute $\sE_i^{(a+1)} * \sE_j * \sE_i^{(b)}$ one looks at 
$$\sE_i * \sE_i^{(a)} * \sE_j * \sE_i^{(b)} \cong \sE_i^{(a+1)} * \sE_j * \sE_i^{(b)} \otimes_\k H^\star(\p^a).$$ 
By induction the left hand side is 
$$\sE_i * \left( \sE_i^{(a+b)} * \sE_j \otimes_\k H^\star(\bG(b,a+b-1)) \oplus \sE_j * \sE_i^{(a+b)} \otimes_\k H^\star(\bG(a,a+b-1)) \right).$$
Now we have 
$$\sE_i * \sE_i^{(a+b)} * \sE_j \cong \sE_i^{(a+b+1)} * \sE_j \otimes_\k H^\star(\p^{a+b})$$
and by Proposition \ref{prop:E1E2E1}
$$\sE_i * \sE_j * \sE_i^{(a+b)} \cong \sE_i^{(a+b+1)} * \sE_j \oplus \sE_j * \sE_i^{(a+b+1)} \otimes_\k H^\star(\p^{a+b-1})$$
So by the cancellation law (\ref{eq:B}), the induction step comes down to proving that 
$$H^\star(\bG(b,a+b-1) \times \p^{a+b}) \oplus H^\star(\bG(a,a+b-1)) \cong H^\star(\bG(b,a+b) \times \p^a)$$
and
$$H^\star(\bG(a,a+b-1) \times \p^{a+b-1}) \cong H^\star(\bG(a+1,a+b) \times \p^a)$$
which one can prove by standard techniques. 
\end{proof}

\subsection{Induced maps}

\begin{Lemma}\label{lem1} If $i,j \in I$ are connected by an edge then 
$$\sE_{ij} * \sE_i \cong \Cone(\sE_i^{(2)} * \sE_j [-1] \xrightarrow{T_{ij}^{(2)(1)}[1]} \sE_j * \sE_i^{(2)}[1])$$
and 
$$\sE_{ij} * \sE_j \cong \Cone(\sE_i * \sE_j^{(2)}[-2] \xrightarrow{T_{ij}^{(1)(2)}} \sE_j^{(2)} * \sE_i).$$
In particular, 
$$\sE_i * \sE_j * \sE_i [-1] \xrightarrow{T_{ij}I} \sE_j * \sE_i * \sE_i$$ 
induces an isomorphism on the $\sE_j * \sE_i^{(2)}[-1]$ summand while
$$\sE_i * \sE_j * \sE_j [-1] \xrightarrow{T_{ij}I} \sE_j * \sE_i * \sE_j$$
induces an isomorphism on the $\sE_i * \sE_j^{(2)}$ summand. We also have the analogous results for $\sE_i * \sE_{ij}$ and $\sE_j * \sE_{ij}$. 
\end{Lemma}
\begin{proof}
We deal with the case of $\sE_{ij} * \sE_i$ since the other cases follow similarly. 

Now $(\sE_i * \sE_j [-1] \xrightarrow{T_{ij}} \sE_j * \sE_i) * \sE_i$ induces a map 
$$\left( \begin{matrix} \alpha & 0 \\ \beta & \gamma \end{matrix} \right): \left( \begin{matrix} \sE_j * \sE_i^{(2)} [-1] \\ \sE_i^{(2)} * \sE_j [-1] \end{matrix} \right) \rightarrow \left( \begin{matrix} \sE_j * \sE_i^{(2)} [-1] \\ \sE_j * \sE_i^{(2)}[1] \end{matrix} \right).$$
We need to show that $\alpha \ne 0 \ne \gamma$ because then $\alpha$ is a non-zero multiple of the identity and by the cancellation Lemma \ref{lem:cancel} the cone is isomorphic to $\Cone(\gamma)$ where $\gamma$ must be $T_{ij}^{(2)(1)}[1]$ (up to a multiple) by Lemma \ref{lem:Eij}.  

Let $ v \in \h'_\k $ be a vector with $ \langle v, \alpha_i \rangle = 1$ and $ \langle v, \alpha_j \rangle = -1 $.

Denote by $t$ the natural inclusion 
$$t: Y(\l) \times Y(\l + \alpha_i + \alpha_j) \rightarrow \tY(\l)|_{\spn(v)} \times_{\spn(v)} \tY(\l+\alpha_i+\alpha_j)|_{\spn(v)}.$$

From Proposition \ref{th:propc1}.(i), for all $ \sA \in Y(\l) \times Y(\l + \alpha_i+ \alpha_j) $, we have the functorial distinguished triangle
$$
\sA[-1] \xrightarrow{c_{v,v}(\sA)} \sA[1] \rightarrow t^* t_* \sA \rightarrow \sA.
$$
Applying this to the distinguished triangle $ \sE_i * \sE_j[-1] \xrightarrow{T_{ij}} \sE_j * \sE_i \rightarrow \sE_{ij} $ we obtain the commutative diagram

\begin{equation}\label{eq:14}
\xymatrix{
t^* t_* (\sE_i * \sE_j [-1]) \ar[rr]^-{t^* t_* T_{ij}} \ar[d]^{(adj)} && t^* t_* (\sE_j * \sE_i) \ar[d]^{(adj)} \ar[r] & t^* t_* \sE_{ij} \ar[d]^{(adj)} \\
\sE_i * \sE_j [-1] \ar[d]^{c_{v,v}(\sE_i *\sE_j)} \ar[rr]^-{T_{ij}} && \sE_j * \sE_i \ar[d]^{c_{v,v}(\sE_j*\sE_i)[1]} \ar[r] & \sE_{ij} \ar[d]^{c_{v,v}(\sE_{ij})} \\
\sE_i * \sE_j [1] \ar[rr]^-{T_{ij}[2]} && \sE_j * \sE_i [2] \ar[r]^{g'} & \sE_{ij}[2]. }
\end{equation}
As noted in section \ref{se:deform}, because $ \la v, \alpha_i + \alpha_j \ra = 0$, $ c_{v,v}(\sE_{ij}) = 0 $.

Now apply $* \sE_i$ to the whole diagram to get 
\begin{equation}\label{eq:15}
\xymatrix{
t^* t_* (\sE_i * \sE_j [-1])*\sE_i \ar[rr]^-{t^* t_* T_{ij}I} \ar[d]^{(adj)} && t^* t_* (\sE_j * \sE_i) *\sE_i \ar[d]^{(adj)} \ar[r] & t^* t_* \sE_{ij}*\sE_i \ar[d]^{(adj)} \\
\sE_i * \sE_j *\sE_i [-1] \ar[d]^{c_{v,v}(\sE_i *\sE_j)I} \ar[rr]^-{T_{ij}I} && \sE_j * \sE_i *\sE_i \ar[d]^{c_{v,v}(\sE_j*\sE_i)I[1]} \ar[r] & \sE_{ij}*\sE_i \ar[d]^{0} \\
\sE_i * \sE_j*\sE_i [1] \ar[rr]^-{T_{ij}I[2]} && \sE_j * \sE_i *\sE_i[2] \ar[r]^{g'I} & \sE_{ij} * \sE_i[2]}. 
\end{equation}

We now examine the map 
$$c_{v,v}(\sE_j*\sE_i)I[1]: \sE_j * \sE_i * \sE_i \cong \sE_j * \sE_i^{(2)} [-1] \oplus \sE_j * \sE_i^{(2)} [1] \rightarrow \sE_j * \sE_i^{(2)} [1] \oplus \sE_j * \sE_i^{(2)} [3] \cong \sE_j * \sE_i * \sE_i [2].$$
We claim that this map is an isomorphism on the summand $\sE_j * \sE_i^{(2)} [1]$. 

To see this, note that $c_{v,v}(\sE_j* \sE_i) I[1] = c_{v,0}(\sE_j * \sE_i) I[1] + c_{0,v}(\sE_j * \sE_i)I[1] $ by Proposition \ref{th:propc1}.(ii).  Let us consider each of these terms.

First, by Proposition \ref{th:propc2}.(i), $c_{v,0}(\sE_j* \sE_i) I = c_{v,0}(\sE_j* \sE_i * \sE_i) $.  Hence this map is given by the diagonal matrix
\begin{eqnarray*}
\left( 
\begin{matrix}
c_{v,0}(\sE_j * \sE_i^{(2)})[-1] & 0 \\
0 & c_{v,0}(\sE_j * \sE_i^{(2)})[1]
\end{matrix}
\right): 
\left( \begin{matrix} \sE_j * \sE_i^{(2)} [-2] \\ \sE_j * \sE_i^{(2)} \end{matrix} \right) \rightarrow
\left( \begin{matrix} \sE_j * \sE_i^{(2)} \\ \sE_j * \sE_i^{(2)} [2] \end{matrix} \right)
\end{eqnarray*}
because the deformation is only along the left-hand factor. In particular, it induces the zero map between the summands $ \sE_j * \sE_i^{(2)}$.

On the other hand, since $ \langle v, \alpha_i \rangle  = 1$, we see by Proposition \ref{th:deformsum} that $ c_{0,v}(\sE_j * \sE_i)I = I c_{0,v}(\sE_i) I $ gives an isomorphism between the $ \sE_j * \sE_i^{(2)} $ summands.  

Hence we conclude that $ c_{v,v}(\sE_j*\sE_i)I[1] $ is an isomorphism between the $\sE_j * \sE_i^{(2)} [1]$ summands.

Finally, looking back at diagram (\ref{eq:15}), we see that $(g'I) \circ c_{v,v}(\sE_j*\sE_i)I[1]  = 0$ so that $c_{v,v}(\sE_j*\sE_i)I[1] $ must factor as
$$\sE_j * \sE_i * \sE_i \rightarrow \sE_i * \sE_j * \sE_i [1] \xrightarrow{T_{ij}I[2]} \sE_j * \sE_i * \sE_i [2].$$
Since $c_{v,v}(\sE_j*\sE_i)I[1] $ induces an isomorphism on the summand $\sE_j * \sE_i^{(2)} [1]$ then so must $T_{ij}I[2]$.  Thus $T_{ij} I$ induces an isomorphism on the summand $\sE_j * \sE_i^{(2)}[-1]$.  This concludes the proof that $\alpha \ne 0$. 

It remains to show $\gamma \ne 0$. To see this we consider the map
$$\sE_i * \sE_i * \sE_j  * \sE_i[-1] \xrightarrow{I T_{ij} I} \sE_i * \sE_j * \sE_i * \sE_i.$$

We will examine this map in two ways, by associating in two different ways.  In particular, we will obtain a contradiction by examining the number of summands of the form $ \sE_i^{(3)} * \sE_j [\cdot] $ on which this map induces an isomorphism. 

On the one hand, we consider the last three factors together and obtain a map 
$$ \sE_i * (\sE_j * \sE_i^{(2)} [-1] \oplus   \sE_i^{(2)} * \sE_j[-1] \rightarrow \sE_j * \sE_i^{(2)} [-1] \oplus \sE_j * \sE_i^{(2)}[1])$$
which is $ I \Bigl( \begin{smallmatrix} \alpha & 0 \\ \beta &\gamma \end{smallmatrix} \Bigr)  $. 

Suppose $\gamma = 0$. Then this map induces an isomorphism on at most one summand of the form $ \sE_i^{(3)} * \sE_j [\cdot]$. This is because it can only induce an isomorphism on a summand coming from $ \sE_i * \sE_j * \sE_i^{(2)} $ and this summand contains one copy of $ \sE_i^{(3)} * \sE_j $ by Proposition \ref{prop:E1E2E1}.

On the other hand, we consider the first three factors together and obtain a map
$$  (\sE_i^{(2)} * \sE_j  \oplus \sE_i^{(2)} * \sE_j [-2]\rightarrow  \sE_i^{(2)} * \sE_j  \oplus \sE_j * \sE_i^{(2)} )* \sE_i $$
We can apply the above reasoning to this map as well since we can write it as$\Bigl( \begin{smallmatrix} \alpha' & 0 \\ \beta' &\gamma' \end{smallmatrix} \Bigr) I $.  Now since $ \alpha' \ne 0 $, we see that this map induces an isomorphism on at least two summands of the form $ \sE_i^{(3)} * \sE_j [\cdot]$ (since $ \sE_i^{(2)} * \sE_j * \sE_i $ contains two copies of $ \sE_i^{(3)} * \sE_j $).  This gives a contradiction and means that $\gamma \ne 0$ (so we are done). 
\end{proof}

\begin{Corollary}\label{cor1} If $i,j \in I$ are connected by an edge then for $s \ge 0$
$$\sE_{ij} * \sE_i^{(s)} \cong \Cone(\sE_i^{(s+1)} * \sE_j [-1] \xrightarrow{T_{ij}^{(s+1)(1)}[s]} \sE_j * \sE_i^{(s+1)} [s]).$$
In particular, 
$$\sE_i * \sE_j * \sE_i^{(s)} [-1] \xrightarrow{T_{ij}I} \sE_j * \sE_i * \sE_i^{(s)} $$ 
induces an isomorphism on all summands of the form $ \sE_j * \sE_i^{(s+1)} [\cdot] $ on the left hand side. Similarly
$$\sE_{ij} * \sE_j^{(s)} \cong \Cone(\sE_i * \sE_j^{(s+1)}[-s-1] \xrightarrow{T_{ij}^{(1)(s+1)}} \sE_j^{(s+1)} * \sE_i).$$
We also have the analogous results for $\sE_i^{(s)} * \sE_{ij}$ and $\sE_j^{(s)} * \sE_{ij}$. 
\end{Corollary}
\begin{Remark} The proof only assumes the result when $s=1$ (everything else is a formal consequence of the fact that $\sE_i^{(a)} * \sE_j^{(b)}$ are bricks). 
\end{Remark}
\begin{proof}

We prove only the first identity as the others follow similarly. 

{\bf Step 1.} First we show by induction on $s$ that 
\begin{eqnarray}\label{eq:temp}
\sE_{ij} * \sE_i^{(s)} \cong \Cone(\sE_i^{(s+1)} * \sE_j [-1] \xrightarrow{g} \sE_j * \sE_i^{(s+1)} [s])
\end{eqnarray}
for some map $g$. The base case $s=1$ is covered in Lemma \ref{lem1}. Consider 
$$(\sE_i * \sE_j [-1] \xrightarrow{T_{ij}} \sE_j * \sE_i) * \sE_i^{(s+1)}$$ 
which we can rewrite as
\begin{equation}\label{eq1}
\sE_j * \sE_i^{(s+2)} \otimes_\k H^\star(\p^s)[-1] \oplus \sE_i^{(s+2)} * \sE_j[-1] \xrightarrow{f_1} \sE_j * \sE_i^{(s+2)} \otimes_\k H^\star(\p^{s+1}).
\end{equation}
Let $ t $ be the number of summands of the form $ \sE_j * \sE_i^{(s+2)} [\cdot]$ on which $ f_1 $ induces an isomorphism.

On the other hand we also have the map 
\begin{equation} \label{eq:cor413}
(\sE_i * \sE_j [-1] \xrightarrow{T_{ij}} \sE_j * \sE_i) * \sE_i^{(s)} * \sE_i.
\end{equation}
Since $ \sE_i^{(s)} * \sE_i \cong \sE_i^{(s+1)} \otimes H^\star(\p^{s}) $ we see that (\ref{eq:cor413}) induces an isomorphism on $t(s+1)$ summands of the form $ \sE_j * \sE_i^{(s+2)} [\cdot] $.

Now we can rewrite (\ref{eq:cor413}) as
$$\left( \sE_j * \sE_i^{(s+1)} \otimes_\k H^\star(\p^{s-1})[-1] \oplus \sE_i^{(s+1)} * \sE_j [-1] \xrightarrow{f_2} \sE_j * \sE_i^{(s+1)} \otimes_\k H^\star(\p^{s}) \right) * \sE_i.$$
By induction, $f_2$ induces an isomorphism on $s$ summands of the form $\sE_j * \sE_i^{(s+1)} [\cdot] $. Now we can rewrite both sides as 
\begin{eqnarray*}
\sE_j * \sE_i^{(s+2)} \otimes_\k H^\star(\p^{s-1} \times \p^{s+1})[-1] \oplus \sE_j * \sE_i^{(s+2)} [-1] \oplus \sE_i^{(s+2)} * \sE_j \otimes_\k H^\star(\p^s)[-1] \\
\xrightarrow{f_3} \sE_j * \sE_i^{(s+2)} \otimes_\k H^\star(\p^s \times \p^{s+1}).
\end{eqnarray*}
The map $f_3$ induces an isomorphism on either $s(s+2)$ or $s(s+2)+1$ summands of the form $ \sE_j * \sE_i^{(s+2)} [\cdot] $  (we do not know {\em a priori} if it induces an isomorphism on the middle summand on the left hand side). 

Combining with above, we see that $t(s+1) = s(s+2) $ or $ t(s+1) = s(s+2) + 1 $ for some $ t $ with $0 \le t \le s+1$.  This forces $t=s+1$.  Hence by Gaussian elimination,
$$\sE_{ij} * \sE_i^{(s+1)} \cong \Cone(\sE_i^{(s+2)} * \sE_j [-1] \xrightarrow{g} \sE_j * \sE_i^{(s+2)} [s+1])$$
for some map $g$. This completes the induction. 

{\bf Step 2.} Next we show that the map $g$ in (\ref{eq:temp}) is non-zero since then by Lemma \ref{lem:Eij}, $g = T_{ij}^{(s+1)(1)}[s]$ (up to a non-zero multiple). If $g=0$ then applying $* \sE_i$ to (\ref{eq:temp}) we get $s+1$ summands 
$$\sE_{ij} * \sE_i^{(s+1)} \cong \Cone(\sE_i^{(s+2)} * \sE_j[-1] \rightarrow \sE_j * \sE_i^{(s+2)}[s+1])$$ 
on the left hand side and 
$$\sE_i^{(s+1)} * \sE_j * \sE_i \oplus \sE_j * \sE_i^{(s+1)} * \sE_i [s]$$
on the right hand side. But then the right side contains $s+3$ summands $\sE_j * \sE_i^{(s+2)}$ instead of $s+1$ on the left side (contradiction). Thus $g \ne 0$. 
\end{proof}

\begin{Corollary}\label{cor:E1E2E1iso} If $i,j \in I$ are connected by an edge then the composition 
$$\sE_i^{(s+1)} * \sE_j \xrightarrow{\i I} \sE_i * \sE_i^{(s)} [-s] * \sE_j \xrightarrow{I T_{ij}^{(s)(1)}} \sE_i * \sE_j * \sE_i^{(s)}$$
is an isomorphism of $\sE_i^{(s+1)} * \sE_j$ onto the lone summand in 
$$\sE_i * \sE_j * \sE_i^{(s)} \cong \sE_i^{(s+1)} * \sE_j \oplus \sE_j * \sE_i^{(s+1)} \otimes_\k H^\star(\p^s).$$
\end{Corollary}
\begin{proof}
Since $\i I$ is an inclusion (into lowest cohomological degree) it suffices to show that 
\begin{equation*}
\sE_i^{(s+1)} * \sE_j \otimes_\k H^\star(\p^{s}) [-s] \cong \sE_i * \sE_i^{(s)} [-s] * \sE_j \xrightarrow{I T_{ij}^{(s)(1)}} \sE_i * \sE_j * \sE_i^{(s)} \cong \sE_i^{(s+1)} * \sE_j \oplus \sE_j * \sE_i^{(s+1)} \otimes_\k H^\star(\p^s)
\end{equation*}
is an isomorphism onto the one copy of $\sE_i^{(s+1)} * \sE_j$ on the right hand side. 

Now consider the map
\begin{equation*}
\sE_i * \sE_j * \sE_i^{(s-1)} [-1] \xrightarrow{T_{ij} I} \sE_j * \sE_i * \sE_i^{(s-1)}.
\end{equation*}
By Corollary \ref{cor1}, on the corresponding summands of both sides, this map restricts to
$$\sE_i^{(s)} * \sE_j [-1] \xrightarrow{T_{ij}^{(s)(1)}[s-1]} \sE_j * \sE_i^{(s)}[s-1].$$
So applying $\sE_i *$ it suffices to show that  
\begin{equation}\label{eq:20}
\sE_i * \sE_i * \sE_j * \sE_i^{(s-1)} [-1] \xrightarrow{I T_{ij} I} \sE_i * \sE_j * \sE_i * \sE_i^{(s-1)}
\end{equation}
is an isomorphism onto all copies of $\sE_i^{(s+1)} * \sE_j$ on the right hand side. 

Now, the right hand side of (\ref{eq:20}) is isomorphic to $\sE_i^{(2)} * \sE_j * \sE_i^{(s-1)} \oplus \sE_j * \sE_i^{(2)} * \sE_i^{(s-1)}$ and (by Lemma \ref{lem1}) the map induces an isomorphism onto the first summand. Since all copies of $\sE_i^{(s+1)} * \sE_j$ on the right hand side of $(\ref{eq:20})$ come from this first summand the map in $(\ref{eq:20})$ surjects onto all summands $\sE_i^{(s+1)} * \sE_j$ on the right side. 
\end{proof}

\begin{Corollary}\label{cor3} If $i,j \in I$ are joined by an edge then 
\begin{eqnarray*}
\Ext^k(\sE^{(a)}_i * \sE_{ij}, \sE^{(a)}_i * \sE_{ij}) \cong
\begin{cases}
0 & \text{if $k < 0$} \\
\k \cdot \id & \text{if $k = 0$.}
\end{cases}
\end{eqnarray*}
Here we are assuming that $\sE^{(a)}_i * \sE_{ij} \ne 0$.
\end{Corollary}
\begin{proof}
By Corollary \ref{cor1} we have the following exact triangle
$$\sE_i^{(a+1)} * \sE_j [-a-1] \xrightarrow{T_{ij}^{(a+1)(1)}} \sE_j * \sE_i^{(a+1)} \rightarrow \sE^{(a)}_i * \sE_{ij}.$$
Applying $\Hom(\cdot, \sE_i^{(a)} * \sE_{ij})$ to this triangle we get 
$$\dots \rightarrow \Ext^{k+a}(\sE_i^{(a+1)} * \sE_j, \sE_i^{(a)} * \sE_{ij}) \rightarrow \Ext^k(\sE^{(a)}_i * \sE_{ij}, \sE^{(a)}_i * \sE_{ij}) \rightarrow \Ext^k(\sE_j * \sE_i^{(a+1)}, \sE^{(a)}_i * \sE_{ij}) \rightarrow \dots$$ 
Now, applying $\Hom(\sE_j * \sE_i^{(a+1)}, \cdot)$ to the exact triangle and using Lemma \ref{lem:Eij} it is easy to see that $\Ext^k(\sE_j * \sE_i^{(a+1)}, \sE^{(a)}_i * \sE_{ij}) = 0$ if $k < 0$ and is one dimensional if $k=0$. Similarly, applying $\Hom(\sE^{(a+1)}_i * \sE_j, \cdot)$ we get
\begin{eqnarray*}
\dots &\rightarrow& \Ext^{k+a}(\sE^{(a+1)}_i * \sE_j, \sE_j * \sE_i^{(a+1)}) \rightarrow \Ext^{k+a}(\sE^{(a+1)}_i * \sE_j, \sE^{(a)}_i * \sE_{ij}) \\
&\rightarrow& \Ext^k(\sE^{(a+1)}_i * \sE_j, \sE_i^{(a+1)} * \sE_j) \rightarrow \Ext^{k+a+1}(\sE^{(a+1)}_i * \sE_j, \sE_j * \sE_i^{(a+1)}) \rightarrow \dots 
\end{eqnarray*}
By Lemma \ref{lem:Eij} this means $\Ext^{k+a}(\sE^{(a+1)}_i * \sE_j, \sE^{(a)}_i * \sE_{ij}) = 0$ if $k < 0$ and if $k = 0$ we get 
$$0 \rightarrow \Ext^{a}(\sE^{(a+1)}_i * \sE_j, \sE^{(a)}_i * \sE_{ij}) \rightarrow \k \cdot \id \rightarrow \k \cdot T_{ij}^{(a+1)(1)} \rightarrow 0$$
where the third map is an isomorphism. Thus $\Ext^{k+a}(\sE^{(a+1)}_i * \sE_j, \sE^{(a)}_i * \sE_{ij}) = 0$ for any $k \le 0$. Thus, if $k \le 0$, we get the exact sequence
$$0 \rightarrow \Ext^k(\sE^{(a)}_i * \sE_{ij}, \sE^{(a)}_i * \sE_{ij}) \rightarrow \Ext^k(\sE_j * \sE_i^{(a+1)}, \sE^{(a)}_i * \sE_{ij})$$
and the result follows. 
\end{proof}

\begin{Lemma}\label{lem2} If $i,j \in I$ are connected by an edge and $\la \l, \alpha_i \ra \ge 0$ then 
$$\sE_{ij} * \sF_i(\l) \cong \Cone(\sF_i * \sE_i * \sE_j [-1] \xrightarrow{IT \oplus  \e I} \sF_i * \sE_j * \sE_i \oplus \sE_j [\la \l, \alpha_i \ra + 1]).$$
In particular, 
$$\sE_i * \sE_j * \sF_i [-1] \xrightarrow{T_{ij}I} \sE_j * \sE_i * \sF_i$$
induces an isomorphism on every summand of the form $\sE_j[\cdot]$ on the left hand side. 
\end{Lemma}
\begin{Remark}\label{rem:lem2} This result is a formal consequence of Lemma \ref{lem1} and Corollary \ref{cor:iso2}.
\end{Remark}
\begin{proof} 
First we consider the map
\begin{equation} \label{eq:4}
\sE_i * \sE_j * \sE_i * \sF_i (\l - \alpha_i)[-1] \xrightarrow{T_{ij} I I} \sE_j * \sE_i * \sE_i * \sF_i.
\end{equation}
On the one hand, we can group the first three factors together to obtain
\begin{equation} \label{eq:first3}
(\sE_j * \sE_i^{(2)} [-1]  \oplus \sE_i^{(2)} * \sE_j [-1]\rightarrow \sE_j * \sE_i^{(2)}[-1] \oplus \sE_j * \sE_i^{(2)}[1] ) * \sF_i(\l-\alpha_i)
\end{equation}
The map on the first summands is an isomorphism by Lemma \ref{lem1}.
Using Corollary \ref{cor:iso2} we have
\begin{equation} \label{eq:first3again2}
\sE_j * \sE_i^{(2)} [-1] * \sF_i(\l-\alpha_i) \cong \sE_j * \sE_i [-1] \otimes_\k H^\star(\p^{\la \l,\alpha_i \ra}) \oplus \sF_i * \sE_j * \sE_i^{(2)} [-1].
\end{equation}
So this induces an isomorphism between at least $\la \l, \alpha_i \ra + 1$ summands of the form $\sE_j * \sE_i [\cdot]$. 

On the other hand, using Proposition \ref{prop:FXE} we have
\begin{equation*}
\sE_i * \sF_i (\l-\alpha_i) \cong \O_{\Delta} \otimes H^\star(\p^{\la \l, \alpha_i \ra -1}) \oplus \sF_i * \sE_i
\end{equation*}
where, as before, $H^\star(\p^{-1}) = 0$ by convention. Hence we can rewrite map (\ref{eq:4}) as
\begin{equation} \label{eq:last2}
\sE_i * \sE_j \otimes  H^\star(\p^{\la \l, \alpha_i \ra - 1})[-1] \oplus \sE_i * \sE_j * \sF_i * \sE_i [-1] \xrightarrow{T_{ij} \oplus T_{ij}II}
\sE_j * \sE_i \otimes  H^\star(\p^{\la \l, \alpha_i \ra -1}) \oplus \sE_j * \sE_i * \sF_i * \sE_i \\
\end{equation}
and then as 
\begin{gather*}
\sE_i * \sE_j \otimes  H^\star(\p^{\la \l, \alpha_i \ra-1})[-1] \oplus \sE_j * \sE_i \otimes H^\star(\p^{\la \l + \alpha_j, \alpha_i \ra + 1})[-1] \oplus \sF_i * \sE_i * \sE_j * \sE_i [-1] \rightarrow  \\
\sE_j * \sE_i \otimes  H^\star(\p^{\la \l, \alpha_i \ra-1}) \oplus \sE_j * \sE_i \otimes H^\star(\p^{\la \l, \alpha_i \ra + 1}) \oplus \sF_i * \sE_j * \sE_i * \sE_i.
\end{gather*}
Now $\sF_i * \sE_i * \sE_j * \sE_i$ contains no summand $\sE_j * \sE_i$ since 
\begin{eqnarray*}
& & \Hom(\sE_j * \sE_i, \sF_i(\l+\alpha_i+\alpha_j) * \sE_i * \sE_j * \sE_i) \\
&\cong& \Hom(\sE_i [\la \l+\alpha_i + \alpha_j, \alpha_i \ra + 1] * \sE_j * \sE_i, \sE_i * \sE_j * \sE_i) \\
&\cong& \Hom(\sE_i^{(2)} * \sE_j \oplus \sE_j * \sE_i^{(2)}, \sE_i^{(2)} * \sE_j \oplus \sE_j * \sE_i^{(2)} [-\la \l, \alpha_i \ra -2])
\end{eqnarray*}
vanishes by using Lemma \ref{lem:Eij} and $\la \l, \alpha_i \ra \ge 0$. Hence (\ref{eq:4}) induces an isomorphism between at most $ \la \l + \alpha_j, \alpha_i \ra + 2 = \la \l, \alpha_i \ra + 1 $ summands of the form $\sE_j * \sE_i [\cdot]$.

Combining these two observations, we see that the map in (\ref{eq:4}) induces an isomorphism between exactly $ \la \l, \alpha_i \ra + 1 $ summands of the form $\sE_j * \sE_i [\cdot]$. This means that the map $T_{ij} II$ from (\ref{eq:last2}) also induces an isomorphism on $ \la \l, \alpha_i \ra + 1 $ summands of the form $\sE_j * \sE_i [\cdot]$.

Now, let us consider 
\begin{equation} \label{eq:TijI}
\sE_i * \sE_j * \sF_i(\l)[-1] \xrightarrow{T_{ij}I} \sE_j * \sE_i * \sF_i 
\end{equation}
which we can rewrite as
\begin{equation*}
\sE_j \otimes H^\star(\p^{\la \l, \alpha_i \ra})[-1] \oplus \sF_i * \sE_i * \sE_j[-1] \rightarrow 
\sE_j \otimes H^\star(\p^{\la \l, \alpha_i \ra + 1}) \oplus \sF_i * \sE_j * \sE_i
\end{equation*}
Since the map $T_{ij}II$ from (\ref{eq:last2}) induces an isomorphism on $ \la \l, \alpha_i \ra + 1 $ summands of the form $ \sE_j * \sE_i [\cdot]$ the map from (\ref{eq:TijI}) must also induce an isomorphism on  $ \la \l, \alpha_i \ra + 1 $ summands of the form $\sE_j [\cdot]$.  Hence we can apply Gaussian elimination to conclude that 
\begin{equation*}
\sE_{ij} * \sF_i(\l) \cong \Cone(\sF_i * \sE_i * \sE_j [-1] \xrightarrow{f_1 \oplus f_2} \sF_i * \sE_j * \sE_i \oplus \sE_j [\la \l, \alpha_i \ra + 1])
\end{equation*}
for some maps $f_1, f_2$. Now 
\begin{equation*}
\Hom(\sF_i (\l + \alpha_i + \alpha_j) * \sE_i * \sE_j [-1], \sE_j [\la \l, \alpha_i \ra + 1]) \cong \Hom(\sE_i * \sE_j [-1], \sE_i * \sE_j [-1]) \cong \k
\end{equation*}
which is spanned by the adjunction map $\e I$. Similarly, we have
\begin{eqnarray*}
& & \Hom(\sF_i (\l+\alpha_i + \alpha_j) * \sE_i * \sE_j [-1], \sF_i * \sE_j * \sE_i) \\
&\cong& \Hom(\sE_i * \sE_j [-1], \sE_i * \sF_i * \sE_j * \sE_i [-\la \l, \alpha_i \ra - 2]) \\
&\cong& \Hom(\sE_i * \sE_j [\la \l, \alpha_i \ra + 1], \sF_i(\l+2\alpha_i+\alpha_j) * \sE_i * \sE_j * \sE_i \oplus \sE_j * \sE_i \otimes_\k H^\star(\p^{\la \l + \alpha_i + \alpha_j, \alpha_i \ra + 1})) \\
&\cong& \Hom(\sE_i * \sE_i * \sE_j [2 \la \l, \alpha_i \ra + 4], \sE_i * \sE_j * \sE_i) \oplus \Hom(\sE_i * \sE_j [\la \l, \alpha_i \ra + 1], \sE_j * \sE_i \otimes_\k H^\star(\p^{\la \l, \alpha_i \ra + 2})). 
\end{eqnarray*}
The first term above equals 
$$\Hom(\sE_i^{(2)} * \sE_j [-1] \oplus \sE_i^{(2)} * \sE_j [1], (\sE_i^{(2)} * \sE_j \oplus \sE_j * \sE_i^{(2)}) [-2 \la \l, \alpha_i \ra  - 4])$$
and thus vanishes by Lemma \ref{lem:Eij} since $\la \l, \alpha_i \ra \ge 0$. The second term is one-dimensional. Thus 
\begin{eqnarray}\label{eq:10}
\Hom(\sF_i * \sE_i * \sE_j [-1], \sF_i * \sE_j * \sE_i) \cong \k
\end{eqnarray}
which is spanned by $IT_{ij}$. So it remains to show that $f_1$ and $f_2$ are non-zero. 

To show that $ f_1 \ne 0 $, we look again at the map (\ref{eq:4}).  When we rewrite it as in (\ref{eq:first3}), we know that the map on first summands is an isomorphism.  By (\ref{eq:first3again2}), these first summands contain a copy of $  \sF_i * \sE_j * \sE_i^{(2)}[-1] $.  On the other hand, when we rewrite (\ref{eq:4}) as in (\ref{eq:last2}), we see that this copy of $ \sF_i * \sE_j * \sE_i^{(2)}[-1] $ is a direct summand of $ \sF_i * \sE_i * \sE_j * \sE_i[-1]$.  Thus the map on $ \sF_i * \sE_i * \sE_j * \sE_i[-1] $ must be non-zero.  However, this is precisely $ f_1I $ and hence $ f_1 \ne 0 $.

To show $f_2 \ne 0$ we consider the map
\begin{equation} \label{eq:ITijI}
 \sE_i * \bigl(\sE_i * \sE_j[-1] \xrightarrow{T_{ij}} \sE_j * \sE_i \bigr)* \sF_i (\l)
 \end{equation}
On the one hand we can rewrite this as
\begin{equation*}
\sE_i^{(2)} * \sE_j * \sF_i [-2] \oplus \sE_i^{(2)} * \sE_j * \sF_i \rightarrow \sE_j * \sE_i^{(2)} * \sF_i \oplus \sE_i^{(2)} * \sE_j * \sF_i
\end{equation*}
where the map is an isomorphism on the second summands. Since 
\begin{equation*}
\sE_i^{(2)} * \sE_j * \sF_i (\l) 
\cong \sE_i * \sE_j \otimes_\k H^\star(\p^{\la \l, \alpha_i \ra + 1}) \oplus \sF_i * \sE_i^{(2)} * \sE_j,
\end{equation*}
this means that (\ref{eq:ITijI}) is an isomorphism on at least $\la \l, \alpha_i \ra + 2$ summands of the form $ \sE_i * \sE_j [\cdot] $. On the other hand, as we showed above, (\ref{eq:TijI}) is an isomorphism on $\la \l, \alpha_i \ra + 1$ summands of the form $\sE_j[\cdot]$. Now, above we showed using Gaussian elimination that the map (\ref{eq:TijI}) can be written as a direct sum of a map which is an isomorphism on $\la \l , \alpha_i \ra + 1 $ summands of the form $\sE_j[\cdot]$ and the map $ f_1 \oplus f_2$. But since (\ref{eq:ITijI}) is obtained from (\ref{eq:TijI}) by applying $ \sE_i * $, this shows that
$$\sE_i * \sF_i * \sE_i * \sE_j [-1] \xrightarrow{I f_1 \oplus I f_2} \sE_i * \sF_i * \sE_j * \sE_i \oplus \sE_i * \sE_j [\la \l, \alpha_i \ra+1]$$
must be an isomorphism on precisely one summand of the form $ \sE_i * \sE_j [\cdot] $.  This means that $I f_2$ must induce an isomorphism on the right hand summand $\sE_i * \sE_j$.  In particular, $f_2 \ne 0$. 
\end{proof}

\begin{Lemma} For any $i \ne j \in I$ the composition $\sF_i^{(s)} * \sE_j \cong \sE_j * \sF_i^{(s)}$ is brick.
\end{Lemma}
\begin{proof}
We have 
\begin{eqnarray*}
\Ext^k(\sE_j(\l) * \sF_i^{(s)}, \sE_j(\l) * \sF_i^{(s)}) 
&\cong& \Ext^k(\sF_i^{(s)} * \sE_j(\l+s\alpha_i), \sE_j(\l) * \sF_i^{(s)}) \\
&\cong& \Ext^k(\sE_j(\l)_L * \sF_i^{(s)}, \sF_i^{(s)} * \sE_j(\l+s\alpha_i)_L) \\
&\cong& \Ext^k(\sF_j * \sF_i^{(s)} [-\la \l, \alpha_j \ra - 1], \sF_i^{(s)} * \sF_j [- \la \l+s\alpha_i, \alpha_j \ra - 1]) \\
&\cong& \Ext^k(\sF_j * \sF_i^{(s)}, \sF_i^{(s)} * \sF_j [-s \la \alpha_i, \alpha_j \ra]).
\end{eqnarray*}
The result now follows from Lemma \ref{lem:Eij} if $i$ and $j$ are joined by an edge and from Corollary \ref{cor:Eij} if $i$ and $j$ are not joined. 
\end{proof}

\begin{Corollary}\label{cor2} If $i,j \in I$ are connected by an edge and $\la \l, \alpha_i \ra + s \ge 0$ then 
$$\sE_{ij} * \sF_i^{(s)}(\l) \cong \Cone(\sF_i^{(s)} * \sE_i * \sE_j [-1] \xrightarrow{IT \oplus I \e I \circ \i I I} \sF_i^{(s)} * \sE_j * \sE_i \oplus \sF_i^{(s-1)} * \sE_j [\la \l, \alpha_i \ra + s]).$$
In particular, 
$$\sE_i * \sE_j * \sF_i^{(s)} [-1] \xrightarrow{T_{ij}I} \sE_j * \sE_i * \sF_i^{(s)}$$
induces an isomorphism on every summand of the form $\sF_i^{(s-1)} * \sE_j [\cdot]$ on the left hand side. 
\end{Corollary}
\begin{Remark}\label{rem:cor2} This result is a formal consequence of Lemma \ref{lem2} and Corollary \ref{cor:iso2}.
\end{Remark}
\begin{proof}
The proof is similar to that of Corollary \ref{cor1}.

{\bf Step 1. } First we prove by induction on $s$ that 
\begin{equation}\label{eq:7}
\sE_{ij} * \sF_i^{(s)} \cong \Cone(\sF_i^{(s)} * \sE_i * \sE_j [-1] \xrightarrow{g_1 \oplus g_2} \sF_i^{(s)} * \sE_j * \sE_i \oplus \sF_i^{(s-1)} * \sE_j [\la \l, \alpha_i \ra +s])
\end{equation}
for some maps $g_1, g_2$. The base case $s=1$ is covered by Lemma \ref{lem2}. Now suppose $\la \l, \alpha_i \ra + s + 1  \ge 0$ and consider
\begin{equation}\label{eq:6}
(\sE_i * \sE_j [-1] \xrightarrow{T_{ij}} \sE_j * \sE_i) * \sF_i^{(s+1)}(\l).
\end{equation}
We have 
\begin{eqnarray*}
\sE_i * \sE_j * \sF_i^{(s+1)}(\l) [-1] &\cong& \sF_i^{(s)} * \sE_j \otimes_\k H^\star(\p^{\la \l+\alpha_j, \alpha_i \ra + s+1}) [-1]  \oplus \sF_i^{(s+1)} * \sE_i * \sE_j [-1] \\
\sE_j * \sE_i * \sF_i^{(s+1)}(\l) &\cong& \sF_i^{(s)} * \sE_j \otimes_\k H^\star(\p^{\la \l, \alpha_i \ra + s+1}) \oplus \sF_i^{(s+1)} * \sE_j * \sE_i.
\end{eqnarray*}
We first want to show that the map induced in (\ref{eq:6}) is an isomorphism on all $\la \l, \alpha_i \ra + s+1$ summands $\sF_i^{(s)} * \sE_j$ on the left hand side. If $\la \l, \alpha_i \ra + s + 1 = 0$ we are done since there are no such summands. 

On the other hand, if $\la \l, \alpha_i \ra + s + 1 > 0$ then we also have the map
$$(\sE_i * \sE_j [-1] \xrightarrow{T_{ij}} \sE_j * \sE_i) * \sF_i (\l) * \sF_i^{(s)}(\l+\alpha_i)$$
which induces a map 
$$\left[ \sE_j \otimes_\k H^\star(\p^{\la \l+\alpha_j, \alpha_i \ra + 1}) [-1] \oplus \sF_i * \sE_i * \sE_j [-1] \xrightarrow{f} \sE_j \otimes_\k H^\star(\p^{\la \l, \alpha_i \ra + 1}) \oplus \sF_i * \sE_j * \sE_i \right] * \sF^{(s)}_i(\l+\alpha_i).$$
By Lemma \ref{lem2}, the map $fI$ induces an isomorphism on all the $\la \l, \alpha_i \ra + 1$ summands of the form $\sE_j * \sF_i^{(s)}[\cdot]$ on the left hand side and also induces 
$$\left( \sF_i * \sE_i * \sE_j [-1] \xrightarrow{IT_{ij}} \sF_i * \sE_j * \sE_i \right) * \sF^{(s)}_i(\l+\alpha_i).$$
By induction this induces an isomorphism on all the $\la \l+\alpha_i+\alpha_j, \alpha_i \ra + s+1$ summands of the form $\sF_i * \sF_i^{(s-1)} * \sE_j [\cdot]$ on the left hand side. So in total $fI$ induces an isomophism on 
$$(\la \l, \alpha_i \ra + 1) + s(\la \l+\alpha_i+\alpha_j, \alpha_i \ra + s+1) = (s+1)(\la \l, \alpha_i \ra + s+1)$$
summands of the form $\sF_i^{(s)} * \sE_j [\cdot]$. This shows that (\ref{eq:6}) induces an isomorphism on the $\la \l, \alpha_i \ra + s+1$ summands (which is what we wanted to show). 

Thus, by the cancellation Lemma \ref{lem:cancel}, 
$$\sE_{ij} * \sF_i^{(s+1)} \cong \Cone(\sF_i^{(s+1)} * \sE_i * \sE_j [-1] \rightarrow \sF_i^{(s+1)} * \sE_j * \sE_i \oplus \sF_i^{(s)} * \sE_j [\la \l, \alpha_i \ra +s+1])$$
which completes the induction. 

{\bf Step 2.} Next, one can check 
$$\Hom(\sF_i^{(s)} * \sE_i * \sE_j [-1], \sF_i^{(s)} * \sE_j * \sE_i) \cong \k \cong \Hom(\sF_i^{(s)} * \sE_i * \sE_j [-1], \sF_i^{(s-1)} * \sE_j [\la \l, \alpha_i \ra + s]).$$
To do this one moves the factor $\sF_i^{(s)}$ from the left side to the right side using adjunction, simplifies $(\sE_i^{(s)} * \sF_i^{(s)}) * \sE_j * \sE_i$ and then uses adjunction again (just like in the computation used to prove (\ref{eq:10})). This is a long but straight-forward calculation which we omit. 

{\bf Step 3.} Finally we show that $g_1$ and $g_2$ are non-zero. This implies that $g_1$ must must be $IT_{ij}$ and $g_2$ must be the composition
\begin{eqnarray*}
\sF_i^{(s)} * \sE_i * \sE_j [-1] 
&\xrightarrow{\i I I}& \sF_i^{(s-1)} * \sF_i(\l+s\alpha_i+\alpha_j) [-s+1] * \sE_i * \sE_j [-1] \\
&\xrightarrow{I \e I}& \sF_i^{(s-1)} * \sE_j [\la \l, \alpha_i \ra + s]
\end{eqnarray*}
(up to a non-zero multiple). Recall that $\i$ denotes the unique inclusion of $\sF_i^{(s)}$ into the lowest degree summand of $\sF_i^{(s-1)} * \sF_i$. 

To show $g_1 \ne 0$ we look at $\sE_{ij} * \sF_i^{(s)} * \sF_i$. Now $\sF_i^{(s)} * \sF_i \cong \sF_i^{(s+1)} \otimes_\k H^\star(\p^{s})$ and from the proof of Step 1, $\left( \sE_i * \sE_j [-1] \xrightarrow{T_{ij}} \sE_j * \sE_i \right) * \sF_i^{(s+1)}$ induces an isomorphism on all $\la \l, \alpha_i \ra + s + 1$ summands of the form $\sF_i^{(s)} * \sE_j [\cdot]$ on the left hand side. Thus 
\begin{equation}\label{eq:temp2}
(\sE_i * \sE_j [-1] \xrightarrow{T_{ij}} \sE_j * \sE_i) * \sF_i^{(s)} * \sF_i
\end{equation}
induces an isomorphism on $(s+1)(\la \l, \alpha_i \ra + s + 1)$ such summands. 

On the other hand, (\ref{eq:temp2}) induces an isomorphism on $\la \l, \alpha_i \ra + s$ summands $\sF_i^{(s-1)} * \sE_j * \sF_i$ or equivalently on $s(\la \l, \alpha_i \ra + s)$ summands of the form $\sF_i^{(s)} * \sE_j [\cdot]$ and what is left over is the map from equation (\ref{eq:7}): 
$$\left( \sF_i^{(s)} * \sE_i * \sE_j [-1] \xrightarrow{g_1 \oplus g_2} \sF_i^{(s)} * \sE_j * \sE_i \oplus \sF_i^{(s-1)} * \sE_j [\la \l, \alpha_i \ra +s] \right) * \sF_i.
$$
This means that the map above must induce an isomorphism on $\la \l, \alpha_i \ra +2s+1$ summands of the form $\sF_i^{(s)} * \sE_j [\cdot]$. This is impossible if $g_1 = 0$ since $\sF_i^{(s-1)} * \sE_j * \sF_i \cong \sF_i^{(s)} * \sE_j \otimes_\k H^\star(\p^{s-1})$ contains only $s$ such summands. Thus $g_1 \ne 0$. 

To show that $g_2 \ne 0$ we apply $\sE_i *$ to (\ref{eq:6}). On the one hand we get 
\begin{equation}\label{eq:8}
\left( \sE_i * \sE_i * \sE_j [-1] \xrightarrow{I T_{ij}} \sE_i * \sE_j * \sE_i \right) * \sF_i^{(s)}.
\end{equation}
By Lemma \ref{lem1}, this induces an isomorphism $(\sE_i^{(2)} * \sE_j \xrightarrow{\sim} \sE_i^{(2)} * \sE_j) * \sF_i^{(s)}$ and the map
$$\left( \sE_i^{(2)} * \sE_j [-2] \xrightarrow{T_{ij}^{(1)(2)}} \sE_j * \sE_i^{(2)} \right) * \sF_i^{(s)}.$$
Now one can show, along the same lines as above, that the map $T_{ij}^{(1)(2)}$ induces an isomorphism on every summand of the form $\sF_i^{(s-2)} * \sE_j [\cdot]$ on the left hand side. Thus the map in (\ref{eq:8}) induces an isomorphism on every summand of the form $\sF_i^{(s-2)} * \sE_j$ on the left hand side. On the other hand, if $g_2 = 0$ then the map
$$\sE_i * \left( \sF_i^{(s)} * \sE_i * \sE_j [-1] \xrightarrow{g_1 \oplus g_2} \sF_i^{(s)} * \sE_j * \sE_i \oplus \sF_i^{(s-1)} * \sE_j [\la \l, \alpha_i \ra + s] \right)$$
cannot induce an isomorphism on all summands $\sF_i^{(s-2)} * \sE_j$ on the left hand side (contradiction). So we must have $g_2 \ne 0$. 
\end{proof}

\section{Proof of Main Theorem \ref{thm:main}}\label{sec:proof}

In this section we will assume that $i,j \in I$ are joined by an edge. For convenience we also assume that $\la \l, \alpha_i \ra \ge 0 $ and $ \la \l, \alpha_i + \alpha_j \ra \ge 0 $, since the other cases are similar.

The main idea of the proof is as follows.  We will show that $ \sT_i * \sT_j = \sT_{ij} * \sT_i $ where $ \sT_{ij} $ is an equivalence coming from an $ \sl_2 $ action generated by the kernel $ \sE_{ij} $.  From a similar argument, we will also show that $ \sT_i * \sT_j =  \sT_j * \sT_{ij}  $.  This immediately implies the braid relation.  The kernel $ \sE_{ij} $ should be thought of as a root vector for the root $ \alpha_i + \alpha_j $.  

In order to prove that $ \sT_i * \sT_j = \sT_{ij} * \sT_i $, we will compute $ \sE_{ij} * \sT_i $.  Recall that $\sT_i$ is the convolution of a complex where each term in the complex is of the form $\sT_i^s = \sF_i^{(\la \l, \alpha_i \ra +s)} * \sE_i^{(s)} [-s]$. So to compute $\sE_{ij} * \sT_i$ we first calculate $\sE_{ij} * \sF_i^{(\la \l, \alpha_i \ra +s)}$ (Step 1) which follows directly from Corollary \ref{cor2}. Next we calculate $\sE_{ij} * \sF_i^{(\la \l, \alpha_i \ra +s)} * \sE_i^{(s)}$ (Step 2) which basically follows from Corollary \ref{cor1}. This gives us a simplified expression for $\sE_{ij} * \sT_i^s$. 

Next, in the most difficult step, we put all these terms together and simplify to come up with an expression for $\sE_{ij} * \sT_i$. We compare with a similarly simplified expression for $\sT_i * \sE_j$ (this is much easier to calculate) and conclude that $\sE_{ij} * \sT_i \cong \sT_i * \sE_j$ (Corollary \ref{cor:EijTi=TiEj}).  It then follows by formal arguments that $\sT_{ij} * \sT_i \cong \sT_i * \sT_j$.

\subsection{Step 1: Calculation of $\sE_{ij} * \sF_i^{(\la \l, \alpha_i \ra +s)}$}

The first step is to compute 
$$\sE_{ij} * \sF_i^{(\la \l, \alpha_i \ra +s)} \in D(Y(\l+s\alpha_i) \times Y(\l-\la \l, \alpha_i \ra \alpha_i)).$$
To simplify things we will abuse notation a little and write $d_i^s$ for any map obtained as the composition 
$$\sF_i^{(k)} * \sE_i^{(s)} \xrightarrow{\i\i} \sF_i^{(k-1)} * \sF_i * \sE_i * \sE_i^{(s-1)} \xrightarrow{I \e I} \sF_i^{(k-1)} * \sE_i^{(s-1)}$$
for any $ k \in \mathbb{N} $ (we omit the necessary shifts here to simplify notation).

\begin{Proposition}\label{prop:step1} $\sE_{ij} * \sF_i^{(\la \l, \alpha_i \ra +s)}(\l-\la \l, \alpha_i \ra \alpha_i)$ is isomorphic to the cone of
$$\sF_i^{(\la \l, \alpha_i \ra + s)} * \sE_i * \sE_j [-1] \xrightarrow{IT_{ij} \oplus d_i^1 I} \sF_i^{(\la \l, \alpha_i \ra + s)} * \sE_j * \sE_i \oplus \sF_i^{(\la \l, \alpha_i \ra + s-1)} * \sE_j [s].$$
\end{Proposition}
\begin{proof}
This follows directly from Corollary \ref{cor2} since 
$$\la \l-\la \l, \alpha_i \ra \alpha_i, \alpha_i \ra + (\la \l, \alpha_i \ra + s) = s \ge 0.$$
\end{proof}

\subsection{Step 2: Calculation of $\sE_{ij} * \sF_i^{(\la \l, \alpha_i \ra +s)} * \sE_i^{(s)}$}

The second step is to compute 
$$\sE_{ij} * \sF_i^{(\la \l, \alpha_i \ra +s)} * \sE_i^{(s)}(\l) \in D(Y(\l) \times Y(\l-\la \l, \alpha_i \ra \alpha_i)).$$

\begin{Proposition}\label{prop:step2} $\sE_{ij} * \sF_i^{(\la \l, \alpha_i \ra +s)} * \sE_i^{(s)}(\l)$ is isomorphic to the cone of
$$\sF_i^{(\la \l, \alpha_i \ra + s)} * \sE_i^{(s+1)} * \sE_j [-1] \xrightarrow{IT_{ij}^{(s+1)(1)}[s] \oplus \gamma_s} \sF_i^{(\la \l, \alpha_i \ra + s)} * \sE_j * \sE_i^{(s+1)}[s] \oplus \sE_j * \sF_i^{(\la \l, \alpha_i \ra + s-1)} * \sE_i^{(s)} [s]$$
where $\gamma_s$ is the composition 
\begin{eqnarray*}
\sF_i^{(\la \l, \alpha_i \ra + s)} * \sE_i^{(s+1)} * \sE_j [-1] 
&\xrightarrow{d_i^{s+1} I}& \sF_i^{(\la \l, \alpha_i \ra + s-1)} * \sE_i^{(s)} * \sE_j \\
&\xrightarrow{IT_{ij}^{(s)(1)}[s]}& \sF_i^{(\la \l, \alpha_i \ra + s-1)} * \sE_j * \sE_i^{(s)} [s]. \\ 
\end{eqnarray*}
\end{Proposition}
\begin{proof}
This is a direct consequence of Proposition \ref{prop:step1}. Applying $* \sE_i^{(s)}$ to the main expression in Proposition \ref{prop:step1} we get the two maps
\begin{equation}\label{eq:12}
\sF_i^{(\la \l, \alpha_i \ra + s)} * \sE_i * \sE_j [-1] * \sE_i^{(s)} \xrightarrow{IT_{ij}I} \sF_i^{(\la \l, \alpha_i \ra + s)} * \sE_j * \sE_i * \sE_i^{(s)}
\end{equation}
\begin{equation}\label{eq:13}
\sF_i^{(\la \l, \alpha_i \ra + s)} * \sE_i * \sE_j [-1] * \sE_i^{(s)} \xrightarrow{d_i^1 I I} \sF_i^{(\la \l, \alpha_i \ra + s-1)} * \sE_j [s] * \sE_i^{(s)}.
\end{equation}
By Corollary \ref{cor1}, (\ref{eq:12}) induces an isomorphism on all summands of the form  $\sF_i^{(\la \l, \alpha_i \ra + s)} * \sE_j * \sE_i^{(s+1)} [\cdot]$ on the left hand side and cancelling out these terms leaves
$$\sF_i^{(\la \l, \alpha_i \ra + s)} * \sE_i^{(s+1)} * \sE_j [-1] \xrightarrow{I T_{ij}^{(s+1)(1)}[s]} \sF_i^{(\la \l, \alpha_i \ra + s)} * \sE_j * \sE_i^{(s+1)} [s].$$
Now the map in (\ref{eq:13}) when restricted to the summand $\sF_i^{(\la \l, \alpha_i \ra + s)} * \sE_i^{(s+1)} * \sE_j[-1]$ is by Corollary \ref{cor:E1E2E1iso} the composition 
\begin{eqnarray*} 
\sF_i^{(\la \l, \alpha_i \ra + s)} * \sE_i^{(s+1)} * \sE_j[-1] 
&\xrightarrow{I \i I}& \sF_i^{(\la \l, \alpha_i \ra + s)} * \sE_i * \sE_i^{(s)} * \sE_j [-s-1] \\
&\xrightarrow{d_i^1 T_{ij}^{(s)(1)}}& \sF_i^{(\la \l, \alpha_i \ra + s-1)} * \sE_j * \sE_i^{(s)} [s].
\end{eqnarray*}
Up to multiple this is the same as the map $\gamma_s$ (completing the proof). 
\end{proof}

\subsection{Step 3: Calculation of $\sE_{ij} * \sT_i$}\label{sec:step3}

\subsubsection{Convolutions} First we recall the precise definition of a (right) convolution in a triangulated category (see \cite{GM} section IV, exercise 1). 

Let $ (A_\bullet, f_\bullet) = A_n \xrightarrow{f_n} A_{n-1} \rightarrow \cdots \xrightarrow{f_1} A_0 $ be a sequence of objects and morphisms such that $ f_i \circ f_{i+1} = 0 $. Such a sequence is called a complex. A (right) convolution of a complex $ (A_\bullet, f_\bullet) $ is any object $ B $ such that there exist
\begin{enumerate}
\item objects $ A_0 = B_0, B_1, \dots, B_{n-1}, B_n = B $ and
\item morphisms $ g_i : B_i [-i] \rightarrow A_i $, $ h_i : A_i \rightarrow B_{i-1}[-(i-1)] $ (with $ h_0 = id $)
\end{enumerate}
such that
\begin{equation} \label{eq:distB}
B_i[-i] \xrightarrow{g_i} A_i \xrightarrow{h_i} B_{i-1}[-(i-1)]
\end{equation}
is a distinguished triangle for each $ i $ and $ g_{i-1} \circ h_{i} = f_i $. Such a collection of data is called a Postnikov system. Notice that in a Postnikov system we also have $f_{i+1} \circ g_i = (g_{i+1} \circ h_i) \circ g_i = 0 $ since $h_i \circ g_i = 0$.

The convolution of a complex need not exist nor is it always unique. However, in the case of the complex
$$\dots \xrightarrow{d_i^{s+1}} \sT_i^s(\l) \xrightarrow{d_i^s} \sT_i^{s-1}(\l) \xrightarrow{d_i^{s-1}} \dots \xrightarrow{d_i^1} \sT_i^0(\l)$$
where $\sT_i^s(\l) = \sF_i^{(\la \l, \alpha_i \ra +s)} * \sE_i^{(s)}(\l) [-s]$ we showed in \cite{ckl3} that the right convolution exists, is unique and gives an object $\sT_i(\l)$ which is invertible.

\subsubsection{Calculation}

We denote the partial right convolution 
$$\sT_i^{\le s}(\l) := \mbox{Conv} \left( \sT_i^{s}(\l) \xrightarrow{d_i^{s}} \sT_i^{s-1}(\l) \xrightarrow{d_i^{s-1}} \dots \xrightarrow{d_i^1} \sT_i^0(\l) \right).$$
Subsequently we have a standard exact triangle 
$$\sT_i^{\le s} \rightarrow \sT_i^{\le s+1} \xrightarrow{\pi} \sT_i^{s+1}[s+1].$$

\begin{Proposition}\label{prop:step3} For any $s \ge -1$, $\sE_{ij} * \sT_i^{\le s}$ is isomorphic to 
$$\Cone \left( \sT_i^{\le s+1} * \sE_j [-1] \xrightarrow{IT_{ij}^{(s+1)(1)}[s] \circ \pi} \sE_j * \sF_i^{(\la \l, \alpha_i \ra + s)} * \sE_i^{(s+1)} [s] \right)$$
where the map above is the composition 
\begin{eqnarray*}
\sT_i^{\le s+1} * \sE_j [-1] 
& & \longrightarrow  
\sT_i^{s+1} * \sE_j [s] = \sF_i^{(\la \l+\alpha_j, \alpha_i \ra + s + 1)} * \sE_i^{(s+1)} * \sE_j [-1] \\
&& \xrightarrow{IT_{ij}^{(s+1)(1)}[s]} \sE_j * \sF_i^{(\la \l, \alpha_i \ra + s)} * \sE_i^{(s+1)} [s].
\end{eqnarray*}
\end{Proposition}
\begin{proof}
The proof is by induction on $s$. The base case is when $s = -1$ which follows since $IT_{ij}^{(0)(1)}[-1]$ is an isomorphism and $\sE_{ij} * \sT_i^{\le -1} = 0$. 

Now we will prove the result for $s+1$ assuming it holds for $s$. The key is the following commutative diagram.

\begin{equation}\label{eq:9}
\xymatrix{
\sF_i^{(\la \l, \alpha_i \ra + s+1)} * \sE_i^{(s+2)} * \sE_j [-2] \ar[rr]^-{{\begin{smallmatrix} \gamma_{s+1}[-1] \\ \oplus IT_{ij}^{(s+2)(1)}[s] \end{smallmatrix}}} \ar[d]^{d_i^{s+2} I} && *{\begin{matrix} \sE_j * \sF_i^{(\la \l, \alpha_i \ra + s)} * \sE_i^{(s+1)} [s] \\ \oplus \sE_j * \sF_i^{(\la \l, \alpha_i \ra + s+1)} * \sE_i^{(s+2)} [s] \end{matrix}} \ar[r] \ar[d]^-{I \oplus 0} & \sE_{ij} * \sT_i^{s+1} [s] \ar@{-->}[d]^{f} \\
\sT_i^{\le s+1} * \sE_j[-1] \ar[d] \ar[rr]^-{IT_{ij}^{(s+1)(1)}[s] \circ \pi} && \sE_j * \sF_i^{(\la \l, \alpha_i \ra + s)} * \sE_i^{(s+1)} [s] \ar[d] \ar[r] & \sE_{ij} * \sT_i^{\le s} \\
\sT_i^{\le s+2} * \sE_j[-1] \ar@{-->}[rr]^-{g} && \sE_j * \sF_i^{(\la \l, \alpha_i \ra + s+1)} * \sE_i^{(s+2)} [s+1] & 
}
\end{equation}
The top left square commutes because, by definition, $\gamma_{s+1}[-1] = (IT_{ij}^{(s+1)(1)}[s]) \circ (d_i^{s+2} I)$. The first two rows and two columns are exact triangles -- note that the second column is exact by the cancellation Lemma \ref{lem:cancel}. The maps $f$ and $g$ are to be determined as explained below. 

In general, if one has a commutative square such as the upper left square in (\ref{eq:9}), then one can fill it with some maps $f,g$ making all the squares commute and so that $\Cone(f) \cong \Cone(g)$ (see Proposition 1.1.11 of [BBD]). 

Using the exact sequence $\sT_i^{\le s} \rightarrow \sT_i^{\le s+1} \rightarrow \sT_i^{s+1}[s+1]$ and Lemma \ref{lem:maps} we find that 
\begin{eqnarray*}
& & \Hom(\sT_i^{\le s+2} * \sE_j[-1], \sE_j * \sF_i^{(\la \l, \alpha_i \ra + s + 1)} * \sE_i^{(s+2)} [s+1]) \\
& \cong & \Hom(\sT_i^{s+2} * \sE_j[-1], \sE_j * \sF_i^{(\la \l, \alpha_i \ra + s + 1)} * \sE_i^{(s+2)} [s+1])
\end{eqnarray*}
is one dimensional and hence it must be spanned by $IT_{ij}^{(s+2)(1)} [s+1] \circ \pi$. It is easy to see that $g \ne 0$ and thus $g = IT_{ij}^{(s+2)(1)} [s+1] \circ \pi$ (up to multiple). 

Similarly, by Lemma \ref{lem:maps} we have 
$$\Hom(\sE_{ij} * \sT_i^{s+1}[s], \sE_{ij} * \sT_i^{\le s}) \cong \Hom(\sE_{ij} * \sT_i^{s+1}[s], \sE_{ij} * \sT_i^{s}) \cong \k.$$
Since $f \ne 0$ we find $f = Id_i^{s+1}$ (up to multiple). Subsequently 
$$\sE_{ij} * \sT_i^{\le s+1} \cong \Cone(I d_i^{s+1}) \cong \Cone(f) \cong \Cone(g) \cong \Cone(IT_{ij}^{(s+2)(1)} [s+1] \circ \pi)$$
and the induction is complete. 
\end{proof}

If we take $s \gg 0$ in Proposition \ref{prop:step3} then we find that 

\begin{Corollary}\label{cor:EijTi=TiEj} $\sE_{ij} * \sT_i \cong \sT_i * \sE_j$. 
\end{Corollary} 

The significance of this is that $\sE_{ij}$ is conjugate to $\sE_j$, namely $\sE_{ij} \cong \sT_i * \sE_j * \sT_i^{-1}$. Thus, for any $k \ge 0$ we can define
$$\sE^{(k)}_{ij}(\l) := \sT_i * \sE_j^{(k)}(\l - \la \l, \alpha_i \ra \alpha_i) * \sT_i^{-1} \text{ and } \sF^{(k)}_{ij}(\l) := \sT_i * \sF_j^{(k)}(\l-\la \l, \alpha_i \ra \alpha_i) * \sT_i^{-1}.$$

\begin{Corollary}\label{cor:step3} Assuming both $\sE_{ij}^{(r)}$ and $\sF_{ij}^{(r)}$ are sheaves they generate a geometric categorical $\sl_2$ action where the one parameter defomation of $Y(\l)$ is the restriction of $\tY(\l)$ to the subspace spanned by $\alpha_i+\alpha_j$. Moreover, we have 
$$\sE_{ij}^{(r)} * \sT_i \cong \sT_i * \sE_j^{(r)} \text{ and } \sF_{ij}^{(r)} * \sT_i \cong \sT_i * \sF_j^{(r)}.$$
\end{Corollary}

\begin{Remark}
In most examples one can verify directly that $\sE_{ij}^{(r)}$ and $\sF_{ij}^{(r)}$ are sheaves. 
\end{Remark}

\begin{Lemma}\label{lem:maps} If $k < 0$ and $t \ge s$ or $k = 0$ and $t > s+1$ 
$$\Ext^k(\sT_i^{t} * \sE_j(\l), \sE_j * \sF_i^{(\la \l, \alpha_i \ra + s-1)} * \sE_i^{(s)}(\l)) = 0 \text{ and }
\Ext^k(\sE_{ij} * \sT_i^t, \sE_{ij} * \sT_i^s) = 0.$$
When $k=0$ and $t = s$ or $t = s+1$ then both of these spaces are one-dimensional. 
\end{Lemma}
\begin{proof}
Proposition 5.2 of \cite{ckl3} claims that $\Ext^k(\sT_i^t, \sT_i^s) = 0$ if $t > s+1$ and $k \le 0$. However, what we showed there is a little stronger than that. One has
$$\Ext^k(\sT_i^{t}, \sT_i^{s}) \cong \Ext^k({\sF_i^{(\la \l + \alpha_j, \alpha_i \ra + s)}}_L * \sF_i^{(\la \l + \alpha_j, \alpha_i \ra + t)} * \sE_i^{(t)} [-t], \sE_i^{(s)} [-s])$$
and one can repeatedly use Corollary \ref{cor:iso2} to write ${\sF_i^{(\la \l + \alpha_j, \alpha_i \ra + s)}}_L * \sF_i^{(\la \l + \alpha_j, \alpha_i \ra + t)} * \sE_i^{(t)} [-t+s]$ as a direct sum 
$$\bigoplus_{l \ge 0} \sF_i^{(l)} * \sE_i^{(s+l)} \otimes_\k V_l$$
for some graded vector spaces $V_l$. Then one can rewrite $\bigoplus_{l \ge 0} \Ext^k(\sF_i^{(l)} * \sE_i^{(s+l)} \otimes_\k V_l, \sE_i^{(s)})$ as 
$$\bigoplus_{l \ge 0} \Ext^k(\sE_i^{(s+l)} \otimes_\k V_l, {\sF_i^{(l)}}_R * \sE_i^{(s)}) \cong \bigoplus_{l \ge 0} \Ext^k(\sE_i^{(s+l)}, \sE_i^{(s+l)} \otimes_\k V'_l)$$
for some (other) graded vector spaces $V'_l$. Then the statement proven in Proposition 5.2 in \cite{ckl3} is that each $V'_l$ is supported in degrees 
$$d \le -(2l^2 + 2l(-2b+t-s)+(t-s)^2+2b(b-t+s)-1)$$
where $b = \la \l, \alpha_i \ra + t - 1$. If we view this as a quadratic in $l$ then the discriminant simplifies to give $-4(t-s)^2+8$ (as it happens it is independent of $b$). This is negative if $t > s+1$ which shows that $V_l'$ must lie in negative degrees. If $t=s$ or $t=s+1$ then there is precisely one value of $l$ for which $d$ is non-negative, namely $l=b$, and we show in \cite{ckl3} that $\dim(V'_b) = 1$.  

Now consider 
$$\Ext^k(\sT_i^{t} * \sE_j, \sE_j * \sF_i^{(\la \l, \alpha_i \ra + s-1)} * \sE_i^{(s)}) \cong \Ext^k(\sT_i^{t} * \sE_j, \sF_i^{(\la \l, \alpha_i \ra + s-1)} * \sE_j * \sE_i^{(s)})$$ 
and suppose we are not in the case $k=0$ and $t=s$ or $t=s+1$. Then, by the same argument as above, we get
$$\bigoplus_{l \ge 0} \Ext^k(\sF_i^{(l)} * \sE_i^{(s+l)} * \sE_j \otimes_\k V_l, \sE_j * \sE_i^{(s)} [s]).$$
By adjunction we can move the term $\sF_i^{(l)}$ from the left side to the right side and try to simplify as before. But now, using Corollary \ref{cor:E1E2E1}, we get terms of the form 
$$\sE_i^{(l)} * \sE_j * \sE_i^{(s)} \cong \sE_i^{(l+s)} * \sE_j \otimes_\k H^\star(\bG(s,s+l-1)) \oplus \sE_j * \sE_i^{(l+s)} \otimes_\k H^\star(\bG(l,s+l-1))$$
instead of terms of the form
$$\sE_i^{(l)} * \sE_i^{(s)} * \sE_j \cong \sE_i^{(l+s)} * \sE_j \otimes_\k H^\star(\bG(s,s+l)).$$
In the old case we saw that we end up with terms of the form $\Ext^k(\sE_i^{(l+s)} * \sE_j, \sE_i^{(l+s)} * \sE_j[n])$ where $n < 0$. Since $H^\star(\bG(s,s+l))$ is supported in degrees $-ls \le m \le ls$ whereas $H^\star(\bG(s,s+l-1))$ in degrees $-(l-1)s \le m \le (l-1)s$ (a difference of $s$) and $H^\star(\bG(l,s+l-1))$ in degrees $-l(s-1) \le m \le l(s-1)$ (a difference of $l$) it must be that we end up with terms of the form 
$$\Ext^k(\sE_i^{(l+s)} * \sE_j, \sE_i^{(l+s)} * \sE_j[s][n-s]) \text{ and } \Ext^k(\sE_i^{(l+s)} * \sE_j, \sE_j * \sE_i^{(l+s)}[s][n-l])$$
where $n < 0$. These vanish by Lemma \ref{lem:Eij}. If $k=0$ and $t=s$ or $t=s+1$ then the right hand terms still vanish while the left hand terms also vanish except when $l=b$ when it is one dimensional. 

To deal with $\Ext^k(\sE_{ij} * \sT_i^t, \sE_{ij} * \sT_i^s)$ we take adjoints and work instead with $\Ext^k(\sT_i^t * \sE_{ij}, \sT_i^s * \sE_{ij})$. Now the same argument as above leaves us with 
$$\Ext^k(\sT_i^t * \sE_{ij}, \sT_i^s * \sE_{ij}) \cong \oplus_l \Ext^k(\sE_i^{(s+l)} * \sE_{ij}, \sE_i^{(s+l)} * \sE_{ij} \otimes V'_l)$$
where $V'_l$ is supported in negative degrees unless $k=0$ and $t=s$ or $t=s+1$ and $l=b$ in which case it is one-dimensional in degree zero. The result now follows by Corollary \ref{cor3}. 
\end{proof}

\subsection{Step 4: Proof that $\sT_{ij} * \sT_i \cong \sT_i * \sT_j$}
In analogy with $\sT_i^s(\l)$ we define
$$\sT_{ij}^s(\l) := \sF_{ij}^{(\la \l, \alpha_i+\alpha_j \ra + s)} * \sE_{ij}^{(s)}(\l) [-s] \in D(Y(\l) \times Y(\l - \la \l, \alpha_i+\alpha_j \ra (\alpha_i + \alpha_j))).$$
Notice that $\sT_{ij}^s \cong \sT_i * \sT_j^s * \sT_i^{-1}$ so that 
$$\Hom(\sT_{ij}^{s}, \sT_{ij}^{s-1}) \cong \Hom(\sT_i * \sT_j^{s} * \sT_i^{-1}, \sT_i * \sT_j^{s-1} * \sT_i^{-1}) \cong \Hom(\sT_j^{s}, \sT_j^{s-1}) \cong \k$$
where the last isomorphism follows from Lemma \ref{lem:dunique} (assuming $\sT_{ij}^s \ne 0$ or equivalently $\sT_j^s \ne 0$). We denote this map by $d_{ij}^s$ (it is well defined only up to a non-zero multiple). One can also describe $d_{ij}^s$ as before by the composition 
\begin{eqnarray*}
\sT_{ij}^s(\l)
&\cong& \sF_{ij}^{(\la \l, \alpha_i + \alpha_j \ra + s)} * \sE_{ij}^{(s)} [-s]  \\
&\xrightarrow{\i \i[-s]}& \sF_{ij}^{(\la \l, \alpha_i + \alpha_j \ra + s-1)} * \sF_{ij} [- \la \l, \alpha_i + \alpha_j \ra - s+1] * \sE_{ij} * \sE_{ij}^{(s-1)} [-(s-1)][-s] \\
&\xrightarrow{I \e I}& \sF_{ij}^{(\la \l, \alpha_i + \alpha_j \ra + s-1)} * \sE_{ij}^{(s-1)} [-s+1] \cong \sT_{ij}^{s-1}(\l). 
\end{eqnarray*}

\begin{Proposition}\label{prop:step4} The complex
$$\dots \rightarrow \sT_{ij}^s(\l) \xrightarrow{d_{ij}^s} \sT_{ij}^{s-1}(\l) \xrightarrow{d_{ij}^{s-1}} \dots \xrightarrow{d_{ij}^1} \sT_{ij}^0(\l)$$
has a unique convolution which we denote $\sT_{ij}(\l)$. Moreover, 
$$\sT_{ij} * \sT_i \cong \sT_i * \sT_j.$$
\end{Proposition}
\begin{proof}
The key is the following commutative diagram
\begin{equation*}
\xymatrix{
\dots \ar[r] & \sT_{ij}^s \ar[r]^{d_{ij}^s} \ar[d]^{\cong} & \sT_{ij}^{s-1} \ar[d]^{\cong} \ar[r] & \dots \\
\dots \ar[r] & \sT_i * \sT_j^s * \sT_i^{-1} \ar[r]^{I d_j^s I} & \sT_i * \sT_j^{s-1} * \sT_i^{-1} \ar[r] & \dots
}
\end{equation*}
where one needs to choose appropriate multiples of the vertical isomorphisms. This is a consequence of the fact that $\Hom(\sT_{ij}^s, \sT_i * \sT_j^{s-1} * \sT_i^{-1}) \cong \Hom(\sT_{ij}^s, \sT_{ij}^{s-1}) \cong \k$ and that $d_{ij}^s \ne 0 \ne d_j^s$. Now $\sT_j^\bullet$ has a unique convolution so $\sT_i * \sT_j^\bullet * \sT_i^{-1}$ has a unique convolution and hence so does $\sT_{ij}^\bullet$.  Finally, the commutativity of the diagram also implies that the convolutions must be isomorphic: i.e. $\sT_{ij} \cong \sT_i * \sT_j * \sT_i^{-1}$. 
\end{proof}

\begin{Lemma}\label{lem:dunique} If $\sT_i^s \ne 0$ and $s \ge 1$ then the space of maps $\Hom(\sT_i^s, \sT_i^{s-1})$ is one-dimensional.
\end{Lemma}
\begin{proof}
We must show that 
$$\Hom(\sF_i^{(\la \l, \alpha_i \ra +s)} * \sE_i^{(s)} [-s], \sF_i^{(\la \l, \alpha_i \ra +s-1)} * \sE_i^{(s-1)} [-s+1]) \cong \k.$$
This essentially done in \cite{ckl3} proof of Proposition 5.2. There we show that 
$$\Hom(\sT_i^s[k-1], \sT_i^{s-k}) = 0$$ 
for $k \ge 2$ whereas we are interested in the left side when $k=1$. The argument there shows that when $k=1$ the left side is equal to a direct sum of terms $\Hom(\sE_i^{(a-j+s)}, \sE_i^{(a-j+s)} [l])$ where $a = \l+s-1$, $j = 0, \dots, a$ and $l \le -2(j-a)(j-a+1)$. Thus $l < 0$ and these terms vanish unless $j=a$. If $j=a$ then $l=0$ and the argument shows we get exactly one such term and then 
$$\Hom(\sT_i^s, \sT_i^{s-1}) \cong \Hom(\sE_i^{(s)}, \sE_i^{(s)}) \cong \k$$
where the second isomorphism follows by Lemma \ref{lem:Eij}). 
\end{proof}

\subsection{Step 5: Proof of braid relation}

Proposition \ref{prop:step4} claims that $\sT_{ij} * \sT_i \cong \sT_i * \sT_j$ which follows from the fact that $\sE_{ij} * \sT_i \cong \sT_i * \sE_j$. 

Now the same proof can be used to show that $\sT_j * \sT_{ij} \cong \sT_i * \sT_j$. Namely, Step 1 is a consequence of the analogous version of Corollary \ref{cor2} which computes $\sF_i^{(s)} * \sE_{ij}$ (this in turn can be traced back to follow formally from Lemma \ref{lem1} and Corollary \ref{cor:iso2} -- see Remarks \ref{rem:cor2} and \ref{rem:lem2}). Then Step 2 and 3 follow formally (they also use Lemma \ref{lem1} and there is some vanishing one needs to check which is a formal consequence of the Lie algebra relations). This shows that $\sT_j * \sE_{ij} \cong \sE_i * \sT_j$ and then Step 4 follows as before. 

Putting these two identities together we get the braid relation 
$$\sT_i * \sT_j * \sT_i \cong \sT_j * \sT_{ij} * \sT_i \cong \sT_j * \sT_i * \sT_j.$$

Finally, if $i,j \in I$ are not joined by an edge then any $\sE_i$ or $\sF_i$ commutes with any $\sE_j$ or $\sF_j$. Since $\sT_i$ is build out of $\sE_i$'s and $\sF_i$'s and $\sT_j$ is built out of $\sE_j$'s and $\sF_j$'s we get the commutativity relation $\sT_i * \sT_j \cong \sT_j * \sT_i$. This concludes the proof of Theorem \ref{thm:main}.

\section{Braiding via strong categorical $\g$-actions}

Strong categorical $\g$-actions have been defined by Khovanov and Lauda in \cite{kl1, kl2, kl3} and independently by Rouquier in \cite{Rnew}. Their definitions are very similar though not identical. One should think of the geometric categorical $\g$-action introduced here as a geometric analogue of their definition which is easier to check in practice. 

In \cite{ckl2} we prove that when $\g = \sl_2$ a geometric $\g$-action implies a strong $\g$-action in the sense of Rouquier. There is good reason to believe the same is true for arbitrary (simply-laced) Kac-Moody Lie algebras $\g$. Nevertheless, in this paper we show that the braid relation follows directly from the geometric $\g$-action. 

On the other hand, our proof of Theorem \ref{thm:main} works to show that a strong $\g$-action gives a braid group action.  In fact it seems that not all the axioms of a strong $ \g $ action are needed to obtain the braid group action.  We will now explain this, starting with a simplified version of Rouquier's definition. 

A (simplified) strong categorical $\g$ action consists of
\begin{enumerate}
\item For each weight $\l$ we have a triangulated category $\D(\l)$.  
\item Exact functors $\E_i^{(r)}(\l): \D(\l) \rightarrow \D(\l+r\alpha_i) \text{ and } \F_i^{(r)}(\l): \D(\l+r\alpha_i) \rightarrow \D(\l)$.
\end{enumerate}

\begin{Remark}
We actually need a little more than this. For each pair $ \l, \l'$ there should be a triangulated category $ \D(\l, \l') $ and an additive functor $\Phi: \D(\l, \l') \rightarrow \Hom(\D(\l), \D(\l')) $, denoted $ \sE \mapsto \Phi_\sE $.  We assume that $ \Phi $ commutes with the cohomological shift $[1]$. We further assume that there is an associative monoidal structure $* : \D(\l, \l') \times \D(\l', \l'') \rightarrow \D(\l, \l'')$ such that $ \Phi $ intertwines this operation with the composition of functors. Moreover, if $ \sE \rightarrow \sF \rightarrow \sG \rightarrow \sE[ 1] $ is a distinguished triangle in $ \D(\l, \l') $, then for any $ A \in \D(\l)$, we require that $ \Phi_\sE(A)  \rightarrow \Phi_\sF(A) \rightarrow \Phi_\sG(A) \rightarrow \Phi_\sE(A)[ 1] $ be a distinguished triangle in $ \D(\l') $. 

Finally, we assume that $\D(\l, \l')$ is idempotent complete and that the hom space between any two objects is finite dimensional. This way $\D(\l,\l')$ satisfies the Krull-Schmidt property. Moreover, we assume that for any non-zero $\sE \in \D(\l, \l')$ we have $\sE \cong \sE [k] \Rightarrow k=0$. 
\end{Remark}

We then require the following relations:
\begin{enumerate}
\item For any weight $\l$, $\Hom(\id_{\D(\l)}, \id_{\D(\l)} [l]) = 0$ if $l < 0$ while $\End(\id_{\D(\l)}) = \k \cdot \id$. 
\item 
\begin{enumerate}
\item $\E^{(r)}_i(\l)_R = \F^{(r)}_i(\l) [r(\la \l, \alpha_i \ra + r)]$
\item $\E^{(r)}_i(\l)_L = \F^{(r)}_i(\l) [-r(\la \l, \alpha_i \ra + r)]$.
\end{enumerate}
\item 
$$\E_i \circ \E_i^{(r)}(\l) \cong \E_i^{(r+1)}(\l) \otimes_\k H^\star(\p^r) \cong \E_i^{(r)} \circ \E_i(\l)$$
while $\E_i \circ \E_j \cong \E_j \circ \E_i$ if $i,j \in I$ are not joined by an edge and 
$$\E_i \circ \E_j \circ \E_i \cong \E_i^{(2)} \circ \E_j \oplus \E_j \circ \E_i^{(2)}$$
if $i,j \in I$ are joined by an edge. 
\item If $\la \l, \alpha_i \ra \le 0$ then 
$$\F_i (\l) \circ \E_i (\l) \cong \E_i(\l - \alpha_i) \circ \F_i (\l - \alpha_i) \oplus \id \otimes_\k H^\star(\p^{-\la \l, \alpha_i \ra - 1})$$
while if $\la \l, \alpha_i \ra \ge 0$ then 
$$\F_i (\l - \alpha_i) \circ \E_i (\l - \alpha_i) \cong \E_i(\l) \circ \F_i (\l) \oplus \id \otimes_\k H^\star(\p^{\la \l, \alpha_i \ra - 1}).$$
If $i \ne j \in I$ then $\F_j \circ \E_j \cong \E_i \circ \F_j$. 
\end{enumerate}
along with the following natural transformations (2-morphisms):
\begin{enumerate}
\item $X_i: \E_i(\l) [-1] \rightarrow \E_i(\l) [1]$ for each $i \in I$ and weight $\l$
\item $T_{ij}: \E_i \circ \E_j (\l) [\la \alpha_i, \alpha_j \ra] \rightarrow \E_j \circ \E_i (\l) $ for any $i,j \in I$ and weight $\l$
\end{enumerate}
with relations
\begin{enumerate}
\item For each $i \in I$ the $X_i$'s and $T_{ii}$'s satisfy the nil affine Hecke relations:
\begin{enumerate}
\item $T_{ii}^2 = 0$
\item $(IT_{ii}) \circ (T_{ii} I) \circ (IT_{ii}) = (T_{ii} I) \circ (IT_{ii}) \circ (T_{ii}I)$ as endomorphisms of $ \E_i \circ \E_i \circ \E_i$.
\item $(X_iI) \circ T_{ii} - T_{ii} \circ (I X_i) = I = - (I  X_i) \circ T_{ii} + T_{ii} \circ (X_i I)$ as endomorphisms of $ \E_i \circ \E_i$.
\end{enumerate}
\item If $i \ne j \in I$ are joined by an edge then $T_{ji} \circ T_{ij} = X_iI + IX_j$ as endomorphisms of $\E_i \circ E_j$. 
\end{enumerate}

This list of 2-morphisms and the relations appear in both the Khovanov-Lauda and Rouquier definitions and it turns out they suffice to prove the braid relations. 

\begin{Remark} Relation (i) above is the only finiteness condition we need. It follows formally that the space of maps between any two compositions of $\E$'s and $\F$'s is finite. The reason we needed the stronger finiteness condition (i) in the definition of geometric categorical actions (section \ref{sec:geomcat}) is because in that case we do not require that $\sE^{(r)}_i * \sE_i$ and $\sF_i * \sE_i$ split as a direct sum (the condition is only at the level of cohomology). The argument that they split requires the cancellation property which in turn requires us to know that all maps are finite dimensional. 
\end{Remark}

\begin{Theorem}\label{thm:main2} A categorical strong $\g$-action as defined above gives rise to equivalences $\T_i$ ($i \in I$) satisfying the braid relations.
\end{Theorem}
\begin{proof}
The fact that we have the nil affine Hecke relations means that for each $i \in I$ we have a strong categorical $\sl_2$ action (in the sense of \cite{ckl3}) so we can construct equivalences $\T_i$. What remains is to show that they braid, which we do by running again through the proof of Theorem \ref{thm:main}. 

The more complicated relations among compositions of $\sE$'s and $\sF$'s (such as Propositions \ref{prop:EXE}, \ref{prop:FXE}, \ref{prop:E1E2E1} and Corollary \ref{cor:E1E2E1}) were all formal arguments which work in any abstract (idempotent complete) category. The same goes for the calculations of $\Hom$-spaces (Lemma \ref{lem:Eij} and Corollary \ref{cor:Eij}).

The only place where something more interesting happens is in the proof of Lemma \ref{lem1}. Notice that the $T_{ij}$ from that Lemma and our $T_{ij}$ defined above must be equal (up to non-zero scalars) since $\Hom(\E_i \circ \E_j [-1], \E_j \circ \E_i) \cong \k$ (by Lemma \ref{lem:Eij}). 

The whole proof of Lemma \ref{lem1} comes down to showing that the map 
$$\E_j \circ \E_i^{(2)} [-1] \oplus \E_i^{(2)} \circ \E_j [-1] \cong \E_i \circ \E_j \circ \E_i [-1] \xrightarrow{T_{ij} I} \E_j \circ \E_i \circ \E_i \cong \E_j \circ \E_i^{(2)} [-1] \oplus \E_j \circ \E_i^{(2)}[1]$$
induces an isomorphism on the $\E_j \circ \E_i^{(2)}[-1]$ summand. In \ref{lem1} we use the fact that $\sE_{ij}$ deforms to $\tsE_{ij}$ to show this. In the abstract setting we use instead the relation $T_{ji} \circ T_{ij} = X_iI + IX_j$. 

More precisely, suppose the map does not induce an isomorphism. This means it must induce zero since $\End(\E_j \circ \E_i^{(2)}) = \k \cdot \id$. But then, the composition 
$$\E_i \circ \E_j \circ \E_i [-1] \xrightarrow{T_{ij} I} \E_j \circ \E_i \circ \E_i \xrightarrow{I T_{ii}} \E_j \circ \E_i \circ \E_i [-2] $$
must be zero. On the other hand, pre-composing with $(T_{ji}I) \circ (IT_{ii})$ we get
\begin{eqnarray*}
(IT_{ii}) \circ (T_{ij}I) \circ (T_{ji}I) \circ (IT_{ii}) 
&=& (IT_{ii}) (X_jII + IX_iI) \circ (IT_{ii}) \\
&=& (X_jII) \circ (IT_{ii})^2 + (IT_{ii}) \circ ((IT_{ii}) \circ (IIX_i) + I) \\
&=& IT_{ii}
\end{eqnarray*}
where we use $T_{ii}^2=0$ twice in the last equality. This is non-zero (contradiction). 

This proves Lemma \ref{lem1}. Then Corollaries \ref{cor1} and \ref{cor:E1E2E1iso} follow by formal arguments. Lemma \ref{lem2} is a formal consequence of Lemma \ref{lem1} and Corollary \ref{cor:iso2} and Corollary \ref{cor2} follows from Lemma \ref{lem2} and Corollary \ref{cor:iso2}. 

This brings us up to section \ref{sec:proof}. One can easily check that the arguments there are formal consequences of the Lie algebra relations and results from section \ref{sec:preliminaries}. So the braid relation follows. 
\end{proof}
\begin{Remark} 
In the setting of (non-categorified) quantum groups, there is a braid group action on $ \U_q(\g) $ (constructed by Lusztig) which is compatible with the braid group action on representations.  Hence we would expect there should be a braid group action on the 2-category of Rouquier/Khovanov-Lauda which is compatible with the above action of the braid group on the representations.

From the proof of the main theorem in this paper, we would expect that generators $ \sigma_i $ of this braid group action would obey the following two conditions
\begin{gather} \label{eq:siEj}
\sigma_i(\E_j) = \T_i \circ \E_j \circ \T_i^{-1} = \Cone(\E_i \circ \E_j [-1] \rightarrow \E_j \circ \E_i) \quad \text{ if } i, j \in I \text{ are joined by an edge } \\
\label{eq:siEi}
\sigma_i(\E_i) = \T_i \circ \E_i \circ \T_i^{-1} = \F_i \text{ (up to a shift) }
\end{gather}

In a forthcoming paper, Khovanov-Lauda will construct a braid group action on the 2-categories from \cite{kl1, kl2, kl3} which satisfies (\ref{eq:siEj}) and (\ref{eq:siEi}) above. Our braid group action will then be compatible with theirs. 
\end{Remark}

\end{document}